\documentclass[11pt]{amsart}
\usepackage[english]{babel}
\usepackage[utf8]{inputenc}
\usepackage{amsmath,amssymb, color,enumerate,amssymb,wasysym,mathrsfs,dsfont}
\usepackage[   backend=biber,
style=numeric,
sorting=none,
isbn=false,
maxbibnames=9]{biblatex}

\usepackage{xcolor}
 \usepackage{mathptmx}
 \usepackage[all,cmtip]{xy}
 \usepackage{csquotes}

 \DeclareMathAlphabet{\mathcal}{OMS}{cmsy}{m}{n}

\DeclareSymbolFont{rsfscript}{OMS}{rsfs}{m}{n}
\DeclareSymbolFontAlphabet{\mathrsfs}{rsfscript}

\DeclareSymbolFont{AMSb}{U}{msb}{m}{n}
\DeclareSymbolFontAlphabet{\mathbb}{AMSb}

\DeclareSymbolFont{eufrak}{U}{euf}{m}{n}
\DeclareSymbolFontAlphabet{\gothic}{eufrak}
\DeclareMathOperator{\op}{op}

\newcommand\Imm{\operatorname{Im}}

\newcommand{\Set}{\mathsf{Set}}
\newcommand{\Sym}{\operatorname{Sym}}
\newcommand{\Aut}{\operatorname{Aut}}

\newcommand{\Grp}{\mathsf{Grp}}

\newcommand{\id}{\operatorname{id}}

\newcommand{\RGrp}{\mathsf{RGrp}}

\DeclareMathOperator{\ind}{i}
\DeclareMathOperator{\per}{p}

\newcommand{\indd}[1]{\ind{#1}}
\newcommand{\perr}[1]{\per{#1}}

\usepackage{hyperref}
\usepackage{doi}
\usepackage[capitalise]{cleveref}

\newtheorem{theorem}{Theorem}[section]
\newtheorem{proposition}[theorem]{Proposition}
\newtheorem{lemma}[theorem]{Lemma}
\newtheorem{corollary}[theorem]{Corollary}
\newtheorem{definition}[theorem]{Definition}

\theoremstyle{definition}
\newtheorem{remark}[theorem]{Remark}

\newtheorem{Ex}[theorem]{Example}

\newcommand{\SGrp}{\mathsf{SGrp}}

\begin{document}
  \title{Right groups and the set-theoretic Yang--Baxter  equation}
  \author{A. Albano}
\address[Andrea Albano]{Dipartimento di Matematica  e Fisica ``Ennio De Giorgi'',\linebreak Università del Salento, 73100 Lecce, Italy}
 \email{andrea.albano@unisalento.it}
     \author{A. Facchini}
\address[Alberto Facchini]{Dipartimento di Matematica ``Tullio Levi-Civita'',\linebreak Universit\`a di 
Padova, 35121 Padova, Italy}
 \email{facchini@math.unipd.it}
\thanks{}
\author{M. Mazzotta}
\address[Marzia Mazzotta]{Dipartimento di Matematica e Fisica ``Ennio De Giorgi'', Università del Salento, 73100
Lecce, Italy}
\thanks{}
\email{marzia.mazzotta@unisalento.it}
\thanks{}
\author{P. Stefanelli}
\address[Paola Stefanelli]{Dipartimento di Matematica e Fisica ``Ennio De Giorgi'', Università del Salento, 73100
Lecce, Italy}
\thanks{}
\email{paola.stefanelli@unisalento.it}


      \begin{abstract} 
      In this paper, we provide techniques to obtain left non-degenerate set-theoretic solutions of the Yang--Baxter equation, drawing on the class of right groups.
     To this end, we introduce the new algebraic structures of left~$RG$-se\-mi\-bra\-ces, which include left (cancellative) semibraces as a proper subclass.
      \end{abstract}

    \maketitle


\section*{Introduction} 
The \emph{Yang--Baxter equation} was first introduced by Yang \cite{Ya67} in Theoretical Physics and, independently, by Baxter \cite{Ba72} in Statistical Mechanics. Since then, the search for its solutions has drawn the attention of many researchers and remains an open problem. Due to the complexity of known solutions, Drinfel'd \cite{Dr92} suggested focusing on a simplified subclass. Namely, if $X$ is a set, a map
$r:X\times X\to X\times X$ is a \emph{set-theoretic solution of the Yang--Baxter equation}, briefly a \emph{solution}, if the identity
\begin{align}\tag{YBE}\label{YBE}
\left(r\times\id_X\right)
\left(\id_X\times \, r\right)
\left(r\times\id_X\right)
= 
\left(\id_X\times \, r\right)
\left(r\times\id_X\right)
\left(\id_X\times \, r\right)
\end{align}
is satisfied. We will write $r\left(x, y\right) = \left(\lambda_{x}\left(y\right), \rho_{y}\left(x\right)\right)$, where $\lambda_{x}$ and $\rho_{y}$ are maps from $X$ into itself, for all $x,y\in X$. Then
$r$ is a solution if and only if for all $x,y,z \in X$
    \begin{align}
     &\label{first} \lambda_x\lambda_y(z)=\lambda_{\lambda_x\left(y\right)}\lambda_{\rho_y\left(x\right)}\left(z\right)\tag{Y1}\\
    &  \label{second}\lambda_{\rho_{\lambda_y\left(z\right)}\left(x\right)}\rho_z\left(y\right)=\rho_{\lambda_{\rho_y\left(x\right)}\left(z\right)}\lambda_x\left(y\right)\tag{Y2}\\
      &\label{third}\rho_z\rho_y(x)=\rho_{\rho_z\left(y\right)}\rho_{\lambda_y\left(z\right)}\left(x\right)\tag{Y3}
  \end{align} 
We say that a solution $r$ is \textit{bijective} if $r$ is a bijective map and, in particular, \textit{involutive} if $r^2=\id_{X\times X}$; 
\textit{left non-degenerate} if $\lambda_x \in \Sym_X$, for all $x \in X$; 
\textit{right non-degenerate} if $\rho_x \in \Sym_X$, for all $x \in X$; \emph{non-degenerate} if it is both left and right non-degenerate. It has only recently been proven in \cite{JePi25} that all non-degenerate solutions are bijective.

The theory of braces was introduced in 2007 by Rump \cite{Ru07} as an algebraic approach to study involutive non-degenerate solutions. A \emph{left brace} is a set $B$ endowed with two binary operations, usually denoted by $+$ and $\circ$, such that $(B, +)$ is an abelian group, $(B, +)$ is a group, and the compatibility condition
\begin{align}\label{1}
    a \circ (b+c)=a \circ b-a+a\circ c
\end{align}
holds, for all $a, b, c \in B$, where $-a$ denotes the inverse of any $a$ in $(B, +)$. Later, Guarnieri and Vendramin \cite{GuVe17} introduced \emph{skew left braces} as a generalization of braces, removing the assumption that the group $(B,+)$ is abelian. Any skew left brace yields non-degenerate solutions.\\
Several further generalizations of braces have been proposed in the literature. Among them, we mention \emph{left (cancellative) semibraces}, introduced in \cite{CaCoSt17}. In the definition of skew left brace, instead of requiring $(B,+)$ to be a group, one assumes that $(B,+)$ is a left cancellative semigroup, and the compatibility condition becomes
\begin{align}\label{2}
    a \circ (b+c)=a \circ b+a \circ \left(a^-+c\right),
\end{align}
for all $a, b, c \in B$, where $a^-$ denotes the inverse of $a$ in the group $(B, \circ)$. 
Note that \eqref{1} and \eqref{2} are equivalent in skew left brace structures; hence, every skew left brace is a left (cancellative) semibrace. The motivation for studying these structures lies in the fact that they still produce solutions that are left non-degenerate. A more detailed study of these algebraic structures can be found in \cite{CaCeSt22}.\\
Subsequently, Jespers and Van Antwerpen \cite{JeAr19} relaxed the condition on the additive structure by introducing \emph{left semibraces} and 
assuming that $(B, +)$ is an arbitrary semigroup. These structures do not necessarily determine solutions (see \cite[Theorem 3]{CCSt20-2}). We also mention that, to study left non-degenerate solutions (not necessarily bijective), the authors in \cite{COJeVaVe22} introduced the algebraic structure of~YB-se\-mi\-truss\-es, which forms a subclass of the category of semitrusses as introduced by Brzeziński \cite{Brz18}.

\smallskip

In this paper, we show that the class of right groups yield a wider class of left non-degenerate solutions than those arising from a left (cancellative) semibrace $B$ whose additive structure is necessarily a right group (see \cite[p.~167]{CaCoSt17}). Recall that right groups are semigroups in which all left translations are bijective maps, or, equivalently, they are semigroups isomorphic to the external direct product of a group and a non-empty right zero semigroup (see \cite[Section~1.11, Theorem~1.27]{ClPr61}). To our end, we introduce and investigate the algebraic structure which we call a {\em left $RG$-semibrace}. It is a triplet $(B,+,\circ)$, where both $(B,+)$ and $(B,\circ)$ are right groups and the identity
\begin{align}\label{bjip-}
a\circ (b + c) = a\circ b+(-a+ a\circ c)
\end{align} 
holds, for all $a,b,c\in B$, where $-a+ a\circ c$ is the unique element $x$ of $B$ such that $a+x=a\circ c$. We prove that \eqref{bjip-} can be written as \eqref{2} and hence left (cancellative) semibraces form a  subclass of left $RG$-semibraces.\\ 
By using structural properties of the two right groups $(B,+)$ and $(B,\circ)$ and the proof technique developed in \cite[cf. Theorem 2.15]{DoRySte24}, we determine a left non-degenerate solution $r_B$ associated with $B$. Additionally, we deduce that the operation $+$ can influence the behaviour of $r_{B}$ by providing solutions that satisfy $r_B^{2n+1} = r_B$, with $n$ depending on the additive groupal component of $B$.
In particular, under certain conditions, we obtain cubic solutions, i.e., $r_B^3 = r_B$, or, more specifically, idempotent solutions.

Furthermore, exploiting the natural connection between averaging operators and right groups \cite[cf. Proposition~3.8]{HuhuXing25}, we employ them to obtain various classes of examples of $RG$-semibraces. Formally, we recall that a \emph{left averaging operator} on a group $B$ is a map $\varphi:B \to B$ such that the equality
$\varphi(a)\varphi(b) = \varphi\left( \varphi(a) b \right)$ is satisfied, for all $a,b\in B$. Note that any idempotent endomorphism of a group is an averaging operator on that same group. Finally, we extend the matched product of left semibraces, originally introduced in \cite{CCSt20-1}, to the wider framework of left $RG$-semibraces, and we specifically focus on the semidirect product as a particular case.

\medskip

\section{Left \texorpdfstring{$RG$-}{-}semibraces}

\subsection{Basic notions on right groups}
In this subsection, we provide basic notions on the class of right groups necessary throughout the paper. 

To this end, let us recall that a \emph{right simple semigroup} $(S,\,\cdot)$ is a semigroup such that $a \cdot S=S$, for all $a\in S$.
According to \cite[Lemma~1.26]{ClPr61}, every idempotent of a right simple semigroup $(S, \cdot)$ is a left identity for $S$.

In particular, a class of right simple semigroups is given by {\em right zero semigroups}, i.e., those semigroups $(S,\,\cdot)$ 
such that $a\cdot b=b$, for all $a,b\in~S$. Since the operation $\cdot$ corresponds to the second canonical projection~$\pi_2\colon S\times S\to S$, 
right zero semigroups are those of the form $(S,\pi_2)$ for some set $S$. Similarly, $a\,\pi_1\, b=a$ determines a {\em left zero semigroup}.  
The full subcategory of the category of semigroups whose objects are all right zero semigroups is clearly isomorphic to the category $\Set$ of sets, because every mapping between two right zero semigroups is a semigroup morphism.

\smallskip

The semigroups satisfying the equivalent conditions of he following result are called {\em right groups.} 

\begin{theorem}\label{1.1}  {\rm \cite[Section~1.11, Theorem~1.27]{ClPr61}} The following assertions on a semigroup $(S,\,\cdot)$, $S\ne\emptyset$, are equivalent:
\begin{itemize}
    \item[{\rm(a)}] $S$ is right simple 
and left cancellative.
\item[{\rm(b)}] For every $a,b\in S$, there exists a unique element $x\in S$ such that $a \cdot x = b$.
\item[{\rm(c)}] $S$ is right simple and contains at least an idempotent.
\item[{\rm(d)}] $S$ is isomorphic to the external direct product of 
a group $G$ and a non-empty right zero semigroup $E$.
\end{itemize}
\end{theorem}

\smallskip

Notice that from
\cref{1.1}(d)
$G$ and $E$ are both isomorphic to homomorphic images and subsemigroups of a right group $S$. Namely, if $\sigma_e$ is the idempotent endomorphism of $S$ defined by setting $\sigma_e(s) = s \cdot e$, for all $s \in S$ and every idempotent $e \in S$, then the image of $\sigma_e$ is the subgroup $S \cdot e$ of $S$ that is isomorphic to $G$.
In the same way, if $\ell_s:S\to S$ is the map defined by $\ell_s(a) = s\cdot a$, for all $a,s\in S$, by \cref{1.1}(b), there is an idempotent endomorphism $\sigma$ of the semigroup $S$ defined by 
$\sigma(s)=\ell_s^{-1}(s)$, for every $s\in S$, whose image is the subsemigroup $E$ of $S$.  
Also, if $\Delta\colon S\to S\times S$ and $\mu\colon S\times S\to S$ are the mappings defined by $\Delta(s) = (s,s)$ and $\mu(s,s')=s\cdot s'$, for all $s,s'\in S$, respectively, then 
$
\mu(\sigma_e\times\sigma)\Delta=\id_S,
$ 
for every $e\in E$. It is in this sense that $S$ is a direct product of $G$ and $E$. Note that, in general,  $\mu(\sigma\times\sigma_e)\Delta$ is different from $\id_S$.

\medskip

For any right group $(S,\,\cdot)$ with a fixed idempotent element $0\in S$, there is a congruence $\sim$ on the semigroup $S$ defined by setting $a\sim b$ if and only if $a\cdot 0=b\cdot 0$, for every $a,b\in S$. Notice that the equivalence $\sim$ is compatible with both the operations $\cdot$ and $\backslash$ on $S$, where $\backslash$ is the inverse operation of $\cdot$. 
Note that the congruence $\sim$ does not depend on the choice of the idempotent $0$.
The equivalence class of any $a\in S$ modulo $\sim$ is $a\cdot E:=\{\,a\cdot e\mid e\in E\,\}$, so that there is a partition $\{\,a\cdot E\mid a\in S\,\}$. A complete irredundant set of representatives of the congruence classes of $S$ modulo $\sim$ is the set $S\cdot 0:=\{\,a\cdot 0\mid a\in S\,\}$, which is a subgroup of $(S,\,\cdot)$ with identity $0$. Moreover, $S\cdot 0$ and $S/{\sim}$ are canonically isomorphic groups.

\medskip

Finally, notice that in the full subcategory $\RGrp$ of the category of semigroups whose objects are all right groups, every semigroup morphism also respects the operation $\backslash$, that is, for all right groups $S,S'$ and all semigroup morphism $f\colon S\to S'$, one necessarily has that $f(s_1\backslash s_2)=f(s_1)\backslash f(s_2)$, for every $s_1,s_2\in S$.

\smallskip

\subsection{Basics on left \texorpdfstring{$RG$-}{-}semibraces}
 
In this subsection, we introduce left~$RG$-se\-mi\-bra\-ces and investigate some of their structural properties.
\begin{definition}\label[definition]{def1}
    {\rm A {\em left $RG$-semibrace} is a triplet $(B,+,\circ)$, where $B$ is a set and $+,\circ$ are binary operations on $B$, such that both $(B,+)$ and $(B,\circ)$ are right groups, 
and the identity
\begin{equation} 
a\circ (b+ c) = a\circ b+(-a+ a\circ c)\label{bjip}
\end{equation} 
holds, for all $a,b,c\in B$, where $-a+ a\circ c$ is the unique element $x$ of $B$ such that $a+x=a\circ c$.} 
\end{definition}

Hence, $-a+ a\circ c$ is  $ a\backslash (a \circ c)$, where $\backslash$ is the inverse of the operation $+$ of the right group $(B,+)$.

\begin{lemma}\label[lemma]{EsubsetF}
    If $B$ is a left $RG$-semibrace, then every idempotent element in the right group $(B,\circ)$ is also idempotent in the right group $(B,+)$.
\end{lemma}

\begin{proof}
    Let $e$ be any fixed idempotent in $(B,+)$ and $f$  be any idempotent in $(B,\circ)$. Then $f\circ (e+ e) = f\circ e+(-f+ f\circ e)$ implies that $e+e=e+(-f+ e)$ because $f$ is a left identity in $(B,\circ)$. Now $e$ is a left identity for $(B,+)$, so $e=-f+e$. Therefore,  $f+e=e$. From the decomposition $B=(B+e)+E$, where $B+e$ is a group with respect to the operation $+$ and $E$ is the set of all idempotent elements in $(B,+)$, we can write any element $b\in B$ in a unique way as $b = g_b + e_b$,
where $g_b=b+e\in B+e$ and $e_b\in E$. Here $e_b=b\backslash b$, where $\backslash$ is the inverse operation of the operation $+$ of the right group $(B,+)$. For $b=f$, it follows that $f=(f+e)+e_f=e+e_f=e_f\in E$. This concludes the proof of the lemma.
\end{proof}

\smallskip

In what follows, we will denote by $0$ a fixed idempotent in $(B,\circ)$, hence necessarily idempotent in $(B,+)$ as well.

\bigskip

From \cref{1.1}(d) we have that $B = G + E$ where $G=B+0$ is a subgroup of $(B, +)$ with identity $0$ and $E$ is the set of all idempotents of $(B,+)$. Accordingly, we can write any element $b\in B$ in a unique way as
$$
b = g_b + e_b
$$
with $g_b\in G$ and $e_b\in E$. Here $g_b=b+0$, and $e_b$ is the unique element of $B$ such that $b+e_b=b$. Notice that the opposite $-g_b$ of $g_b$ in the additive group $B+0$ is the unique element $x\in B$ such that $b+x=0$. We have the two projections $B\to G$ and $B\to E$ defined by $b\mapsto g_b=b+0$ and $b\mapsto e_b$, respectively. These projections are semigroup homomorphisms. 
Similarly, we also have $B=H \circ F$ where $H=B \circ 0$ is a subgroup of $(B, \circ)$ with identity $0$ and $F$ is the set of all idempotents of $(B,\circ)$. Thus, any element $b \in B$ can be also written in a unique way as
$$
b = h_b \circ f_b
$$
with $h_b\in H$ and $f_b\in F$. We have that $h_b=b\circ 0$, and $f_b$ is the unique element of $B$ such that $b\circ f_b=b$. We also have the projections $B\to H$ and $B\to F$ defined by $b\mapsto h_b$ and $b\mapsto f_b$, respectively. Hereinafter, in the group $H=B\circ0$, we will denote by $h_b^-$ the inverse of $h_b$. Thus $h_b^-$ is the unique element $y\in B$ such that $b\circ y=0$. 

Also, in any left $RG$-semibrace $B$, we know that $F\subseteq E$ by Lemma \ref{EsubsetF}. As a consequence, if $(B, +)$ is a group, then $(B, +, \circ)$ is a skew left brace.

\smallskip

In these notations, the elements $-a+a\circ c$ and $-g_a + a\circ c$ coincide, for all $a, c\in B$. In fact, $a+(-g_a + a\circ c) = g_a + e_a+( - g_a + a\circ c)$, and because all elements of $E$ are left identities in $(B,+)$, it follows that $e_a +( - g_a + a\circ c)= - g_a + a\circ c$. Therefore $a+(-g_a + a\circ c) = g_a + (- g_a + a\circ c) = a\circ c$. This proves that 
$$
-a+a\circ c=-g_a + a\circ c,
$$ 
for every $a,c\in B$.

\medskip

Consequently, \cref{def1} can be equivalently stated in a different way, as follows.  

\begin{definition}\label[definition]{def2}
{\rm A {\em left $RG$-semibrace} is a triplet $(B,+,\circ)$, where $B$ is a set and $+,\circ$ are binary operations on $B$, such that both $(B,+)$ and $(B,\circ)$ are right groups, there exists an element $0\in B$ that is idempotent
both in $(B,+)$ and $(B,\circ)$, and
\begin{equation}\label{bjip1}
a\circ (b + c) = a\circ b -g_a + a\circ c,
\end{equation} 
for all $a,b,c\in B$.}
\end{definition}

According to \cite[p. 167]{CaCoSt17} and equation (2) in \cite[p. 134]{CaCeSt22}, any left (cancellative) semibrace is a left $RG$-semibrace for which $F\subsetneq E$. In particular, all skew left braces are left $RG$-semibraces. 

Any right group $(B,\cdot)$ determines a \emph{trivial} left $RG$-semibrace if we consider the triplet $(B, \cdot, \cdot)$.

\smallskip

\begin{remark}\label{remtruss} {\rm Let $B$ be a left $RG$-semibrace. 
The assignment $a\mapsto g_a$ can also be seen as an idempotent endomorphism $\sigma\colon B\to B$ of the additive right group $(B,+)$. In this notation, our condition (\ref{bjip1}) in Definition \ref{def2} becomes equivalent to $a\circ (b+ c) = a\circ b -\sigma(a) + a\circ c$, for all $a, b, c \in B$. Thus, our definition of left $RG$-semibrace is strictly related to the definition of \emph{skew left truss} \cite[Section~2]{Brze} since, in this context, the three elements $-a+a\circ c$, $-g_a + a\circ c$ and $-\sigma(a) + a\circ c$ coincide, for all $a, c\in B$.}
\end{remark}

\medskip

As usual in brace theory, we now introduce the lambda map associated with a left $RG$-semibrace $B$, through which identity \eqref{bjip1} can also be rewritten.

\begin{proposition}\label[proposition]{lambda1}
    Let $B$ be a left $RG$-semibrace. Define a map $\lambda_a:B \to B$ by setting 
    $$
    \lambda_a(b) = -g_a + a \circ b,
    $$ 
    for all $a,b\in B$.
    Then the following hold:
    \begin{itemize}
        \item[\rm (a)] 
        $\lambda_a \in \Aut(B,+)$ with inverse given by $\lambda_a^{-1} = \lambda_{h_a^-}$, for every $a\in B$.
           \item[\rm (b)] $\lambda_a(b) = a \circ (h_a^- + b)$, for all $a, b \in B$.
       \item[{\rm(c)}] the map $\lambda:B\to \Aut(B, +)$, $a\mapsto \lambda_a$, is a semigroup homomorphism from $(B, \circ)$ into $\Aut(B, +)$.
     \item[{\rm(d)}] $\lambda_a = \lambda_{h_a}$, for all $a \in B$. 
    \end{itemize}\end{proposition}
    \begin{proof}
        Let $a, b, c \in B$ and first note that
        \begin{align*}
            \lambda_a(b+c) &= -g_a + a \circ (b+c)
            = -g_a + a \circ b -g_a + a \circ c
            = \lambda_a(b) + \lambda_a(c) \,.
        \end{align*}
        Next, for all $a, b \in B$ define $\tau_a(b) = h_a^- \circ (a+b)$. 
        Then, the following hold:
        \begin{align*}
            \lambda_a\tau_a(b) &= -g_a + a \circ \tau_a(b) 
            = -g_a + h_a \circ h_a^- \circ (a+b)
            = b \,, \\
            \tau_a\lambda_a(b) &=
            h_a^- \circ (a + \lambda_a(b))
            = h_a^- \circ (g_a - g_a + a \circ b)
            = b,
        \end{align*}
        for all $a, b \in B$. Moreover, if $a, b \in B$ then we have that
        \begin{align*}
            \lambda_a^{-1}(b) =
            h_a^- \circ (a+b) 
            = h_a^- \circ a -g_{h_a^-} + h_a^- \circ b = -g_{h_a^-} + h_a^- \circ b 
            = \lambda_{h_a^-}(b).
        \end{align*} This concludes the proof of {\rm (a)}.
         Now, let $a, b \in B$ and note that 
        \begin{align*}
            a \circ \left(h_a^-+b\right) 
            &= h_a\circ f_a \circ h_a^--g_a+a \circ b 
            =-g_a+a \circ b, 
        \end{align*}
        thus proving {\rm (b)}. Moreover, we have that
    \begin{align*}
        \lambda_{a\circ b}(c)
        &= a\circ b\circ(h_b^-\circ h_a^- +c)
        = a\circ\left(b\circ h_b^-\circ h_a^- -g_b + b\circ c\right) \\
        &= a\circ\left(h_a^- +\lambda_b(c)\right) = \lambda_a\lambda_b(c),
    \end{align*} 
thus proving {\rm (c)}.
Now, since $\lambda_f = \id_B$, for all $f \in F$, we get
 \begin{align*}
     \lambda_a = \lambda_{h_a \circ f_a} = \lambda_{h_a}\lambda_{f_a} = \lambda_{h_a},
 \end{align*}
 for all $a \in B$, concluding the proof. 
\end{proof}

\smallskip

Notice that $\lambda_a(b)$ is $ a\backslash (a \circ b)$, where $\backslash$ is the inverse of the operation $+$ of the right group $(B,+)$.

\smallskip

\begin{remark}\label{1e2} {\rm Let $\Grp$ denote the category of groups and $\SGrp$ the category of semigroups. 
If $B$ is a left $RG$-semibrace, it is not difficult to show that the group $\Aut_\SGrp(B,+)$ is isomorphic to the direct product $\Aut_\Grp(B+0,+) \times \Sym_E$, where $\Sym_E$ denotes the group of all permutations of the set $E$. 
To this aim, it is sufficient to consider the group homomorphism $$\mathcal{F}\colon\Aut_\SGrp(B,+)\to \Aut_\Grp(B+0,+) \times \Sym_E$$ that associates to every $f\in \Aut_\SGrp(B,+)$ the pair $\mathcal{F}(f)=(f_1,f_2)$, where $f_1\colon B+0\to B+0$ is defined by $f_1(x)=f(x)+0$, for every $x\in B+0$, and $f_2\colon E\to E$ is the restriction of $f$ to $E$. 
The inverse of $\mathcal{F}$ is the map $\mathcal{G}\colon\Aut_\Grp(B+0,+) \times \Sym_E\to \Aut_\SGrp(B,+)$ that associates to each pair $(g,h)\in \Aut_\Grp(B+0,+) \times \Sym_E$ the mapping 
\begin{align*}
    \mathcal{G}(g,h): B\to B \,,\,\, b \mapsto g(b+0)+h(e_b) \,,
\end{align*}
for all $b\in B$. \\
Now, if we put together the following three facts:
\begin{itemize}
    \item[{\rm(1)}] $\lambda\colon (B, \circ)\to \Aut_\SGrp(B, +)$ is a semigroup homomorphism;
    \item[{\rm(2)}] $(B, \circ)=(B\circ 0)\circ F$;
    \item[\rm{(3)}]$\Aut_\SGrp(B,+)\cong\Aut_\Grp(B+0,+) \times \Sym_E$; 
    \end{itemize}
    then we get that $\lambda$ essentially corresponds to the follwing two group morphisms 
    \begin{align*}
        (B\circ 0, \circ) &\to \Aut_\Grp(B+0,+) \,,\,\, b\in B+0\mapsto (\lambda_b)_1 \\
        (B\circ 0, \circ) &\to \Sym_E \,,\,\, b\in B+0\mapsto (\lambda_b)_2 \,.
    \end{align*}
    Hence, $\lambda$ corresponds to an action of the group $(B\circ 0, \circ)$ on the group $(B+0,+)$ and an action of the group $(B\circ 0, \circ)$ on the set $E$.}
\end{remark}

\smallskip

The following shows how the $\lambda$ maps allow us to move from the operation $+$ to the operation $\circ$ and conversely.

\begin{lemma}\label[lemma]{somma_cerchietto}
Let $B$ be a left $RG$-semibrace. Then, for all $a, b \in B$:
    \begin{enumerate}
      \item[\rm (a)] $a \circ b=a+\lambda_a(b)$.
       \item [\rm (b)] $a+b=a \circ \lambda_{h_a^-}(b)$.
    \end{enumerate}
    \begin{proof}
       Let $a, b \in B$. Then we have that $a+\lambda_a(b)=a-g_a+a \circ b=a \circ b.$ Moreover, 
       $$
       a \circ \lambda_{h_a^-}(b) = a \circ h_a^-\circ \left(g_a+b\right)=g_a+b=a+b.
       $$
    \end{proof}
\end{lemma}

\smallskip

\begin{proposition}\label[proposition]{prop_GHEF}
    Let $B$ be a left $RG$-semibrace.
    Then the following hold:
    \begin{itemize}
        \item[{\rm(a)}] $g_a = g_{h_a}$ \ and \ $e_a= \lambda_{a}(f_a)$, for all $a \in B$. In particular, $\lambda_a(0)=e_{h_a}$, for all $a \in B$.
        \item[{\rm(b)}] $\lambda_e\left(H\right) \subseteq H$, for all $e \in E$.
    \end{itemize}
    \begin{proof}
    Let $a \in B$.  By \cref{somma_cerchietto}(a) and \cref{lambda1}(d), we have that
    \begin{align*}
        g_a + e_a = a
        = h_a\circ f_a
        = h_a + \lambda_{h_a}(f_a) 
        = g_{h_a} + \lambda_{a}(f_a). 
    \end{align*}
    Now, $\lambda_{a} \in \Aut(B,+)$ and $\lambda_{a}(f_a)\in E$  by \cref{EsubsetF}. The conclusion follows immediately, thus obtaining the first statement in point {\rm(a)}. 
    In addition,
  $$
\lambda_a(0)=-g_a+g_{h_a}+e_{h_a}=e_{h_a},
  $$ 
  for all $a \in B$.
    Moreover, if we let $e\in E$ and $h\in H$, then 
    by \cref{somma_cerchietto} we deduce that
    \begin{align*}
    \lambda_e(h)=e+\lambda_e(h)=e \circ h=e \circ h \circ 0=\lambda_e(h)\circ 0,
    \end{align*}
    hence $\lambda_e(h)\in H$.
    \end{proof}
\end{proposition}

\smallskip

In the following, if $X$ and $Y$ are subsets of a left RG-semibrace $B$, we denote by $X + Y$ the subsemigroup of $(B, +)$ generated by elements of the form $x+y$, with $x\in X$ and $y\in Y$. 
\begin{proposition}
Let $(B,+,\circ)$ be a left RG-semibrace. Then the following statements hold:
    \begin{enumerate}
        \item[{\rm(a)}] $(E,+,\circ)$ is a left RG-semibrace.
        \item[{\rm(b)}] 
        $(G + H, +, \circ)$ is  a sub-structure of \ $B$ \ and \ $B = (G + H) +  E$.
        \item[{\rm(c)}] $B = H + E= H\circ E$.
        \item[{\rm(d)}] $(G\cap H, \circ)$ is a subgroup of the group $(H,\circ)$.
    \end{enumerate}
\end{proposition}
\begin{proof} (a) It is enough to observe that
        \begin{align*}
        e_1\circ e_2 +  e_1\circ e_2
        = e_1\circ e_2 - 0 +  e_1\circ e_2
            = e_1\circ(e_2 + e_2)
            =  e_1\circ e_2,
        \end{align*}
        for all $e_1,e_2\in E$.
        
        (b) Observe that $G+H$ is closed with respect to $+$ by \cref{prop_GHEF}(a). To prove that $G+H$ is closed with respect to $\circ$ it is enough to observe that, by \cref{somma_cerchietto}(a) and again \cref{prop_GHEF}(a),
        \begin{align*}
            b\circ(g+h)
            &= b\circ g - g_{b} 
            + b\circ h\\
            &= 
            g_b+g^{}_{\lambda_{b}(g)}-g_{b} + h_b\circ h
            \ \in \ G + H\,,
        \end{align*}
        for all $b\in B$, $g\in G$, and $h\in H$.
        
(c) Let $b\in B$. Then, by \cref{somma_cerchietto}(a), \cref{lambda1}(d), and \cref{prop_GHEF}(a), we have that
        \begin{align*}
            b=h_b \circ f_b=h_b+\lambda_{h_b}(f_b)=h_b+\lambda_b(f_b)=h_b+e_b \in H + E.
        \end{align*}
        The second part follows directly from \cref{EsubsetF}.
        
 (d)    Let $x,y\in G\cap H$. Then
\begin{align*}
    x\circ y + 0 
    &= x\circ y - x + x\circ 0 &\mbox{since $x\in G\cap H$}\\
    &= x\circ(y+0)
    = x\circ y&\mbox{since $y\in G$}
\end{align*}
Now, by \cref{prop_GHEF}(a), $\lambda_x(0)=0$, thus,  
by \cref{somma_cerchietto}(b),
\begin{align*}
    x^- + 0 = x^-\circ\lambda_x(0) = x^-\circ 0 = x^-,
\end{align*}
since $x^- \in H$. 
Hence, the claim follows.
\end{proof}

\smallskip

The following result allows one to identify a skew left subbrace structure for a specific class of $RG$-semibraces.

\begin{proposition}\label[proposition]{prop_lambda_g_0}
    Let $(B,+,\circ)$ be a left $RG$-semibrace. Then, the following hold:
    \begin{itemize}
        \item[{\rm (a)}] $ G \cap H = \{ g \in G \mid \lambda_g(0) = 0 \}$.
        \item[{\rm (b)}] $(G,+,\circ)$ is a skew left brace if and only if $\lambda_g(0) = 0 $ holds, for all $g \in G$.
    \end{itemize}
\end{proposition}
\begin{proof}
Set $T = \{ g \in G \mid \lambda_g(0) = 0 \}$ and let us prove first point {\rm (a)}. Note that if $g \in T$ then $g \circ 0 = g +\lambda_g(0) = g$, hence $g \in G \cap H$.
Conversely, if $g \in G \cap H$ then
\begin{align*}
    g + 0 = g = g \circ 0 = g + \lambda_g(0) \,,
\end{align*}
from which follows that $\lambda_g(0) = 0$, thus proving the first point.
To prove point {\rm (b)}, it is sufficient to verify that $G$ is closed under the $\circ$ operation if and only if $\lambda_g(0) = 0$ holds, for all $g \in G$.
To this aim, first assume that the latter condition holds, let $g_1,g_2\in G$ and note that, by \cref{somma_cerchietto}(a), we have
\begin{align*}
  g_1\circ g_2 + 0
  = g_1 + \lambda_{g_1}(g_2) + 0.
\end{align*}
Hence, if we assume that $\lambda_g(0) = 0$ is identically satisfied, for all $g \in G$, then by \cref{somma_cerchietto}(a)
\begin{align*}
  g_1\circ g_2 + 0
  = g_1 + \lambda_{g_1}(g_2 + 0)
  = g_1 + \lambda_{g_1}(g_2),
\end{align*}
namely, $G$ is closed under operation $\circ$.
Conversely, assume that $G$ is closed under $\circ$ and fix $g_1,g_2 \in G$.
Then $g_1\circ g_2 + 0 = g_1\circ g_2$ from which follows that
\begin{align*}
  g_1 + \lambda_{g_1}(g_2) + 0 = g_1 + \lambda_{g_1}(g_2)
\end{align*}
and, by left cancellativity, $\lambda_{g_1}(g_2) + 0 = \lambda_{g_1}(g_2)$. Hence, by \cref{lambda1},
\begin{align*}
  \lambda_{g_1}(g_2 + \lambda^{-1}_{g_1}(0)) = \lambda_{g_1}(g_2)
\end{align*}
and, by the bijectivity of the map $\lambda_{g_1}$, $g_1+\lambda^{-1}_{g_1}(0) = g_2$. Since $g_1,g_2\in G$ and by \cref{lambda1}, $\lambda^{-1}_{g_1}(0)\in E$, we obtain that $\lambda_{g_1}(0) = 0$.
\end{proof}

\medskip

As recalled in the preliminaries, every right group admits a congruence that is independent of the chosen idempotent. We now conclude this section by establishing some properties of this relation in the setting of left $RG$-semibraces.

\begin{remark} 

{\rm For a left $RG$-semibrace $(B,+,\circ)$ we have:

(1) For the right group  $(B,+)$ a congruence $\sim_+$ defined, for every $a,b\in B$, by $a\sim_+ b$ if $a+ 0=b+0$. The equivalence class of any $a\in B$ modulo $\sim_+$ is 
$$
a+E=\{\,a+e\mid e\in E\,\},
$$ 
so that there is a partition $\{\,a+E\mid a\in B\,\}$. A complete irredundant set of representatives of the congruence classes of $B$ modulo $\sim_+$ is $B+ 0=\{\,a+ 0\mid a\in B\,\}$. The groups $(B+0,+)$ and $(B/{\sim_+},+)$ are isomorphic. 

(2) For the right group  $(B,\circ)$ a congruence $\sim_\circ$ defined, for every $a,b\in B$, by $a\sim _\circ b$ if $a\circ 0=b\circ 0$. The equivalence class of every $a\in B$ modulo $\sim_\circ$ is 
$$
a\circ F=\{\,a\circ f\mid f\in F\,\},
$$ 
so that there is a partition $\{\,a\circ F\mid a\in B\,\}$. A complete irredundant set of representatives of the congruence classes of $B$ modulo $\sim_\circ$ is $B\circ 0=\{\,a\circ 0\mid a\in B\,\}$. The groups $(B\circ 0,\circ)$ and $(B/{\sim_\circ},\circ)$ are isomorphic.} 
\end{remark}

\smallskip

\begin{proposition} Let $(B,+,\circ)$ be a left $RG$-semibrace. Then:
\begin{enumerate}
    \item[\rm (a)] $b\circ F\subseteq b+E$, for every $b\in B$.

\item[\rm (b)] The equivalence $\sim_\circ$ is contained in the equivalence $\sim_+$, so that there is a canonical surjective mapping $B/{\sim_\circ}\to B/{\sim_+}$ defined by 
$$
[b]_{\sim_\circ}=b\circ F\mapsto [b]_{\sim_+} = b+E,
$$ 
for every $b\in B$.

\item[\rm (c)] As far as the complete irredundant set of representatives $B+ 0$ and $B\circ 0$ are concerned, there is a canonical injective mapping $B+ 0\to B\circ 0$,  defined by $b+ 0\mapsto (b+0)\circ 0$, for every $b+0\in B+0$.

\item[\rm (d)] The composite mapping of the injective mapping $B+ 0\to B\circ 0$ in {\rm (c)}, the group isomorphism $B\circ 0\to B/{\sim_\circ}$ and the surjective map $B/{\sim_\circ}\to B/{\sim_+}$ in {\rm (b)}, is the canonical the group isomorphism  $B+ 0\to B/{\sim_+}$.
\end{enumerate}\end{proposition}

\begin{proof} (a) If $a\in b\circ F$, then $a=b\circ f$ for some $f\in F$. From \cref{somma_cerchietto}(a) it follows that $a=b+\lambda_b(f)$. Now $\lambda_b$ is an automorphism of $(B,+ )$ and $f\in F\subseteq E$ is an idempotent in $(B,+ )$. Therefore $\lambda_b(f)\in E$. Thus $a=b+\lambda_b(f)\in b+E$.

(b) Follows from (a) and the fact that the equivalence classes modulo $\sim_\circ$ are the sets $b\circ F$ and the equivalence classes modulo $\sim_+$ are the sets $b+E$.

(c) Let $b+0, b'+0$ be two elements of $B+0$ such that $(b+0)\circ 0=(b'+0)\circ 0$. Then $(b+0)\sim_\circ(b'+0)$, so that $(b+0)\sim_+(b'+0)$ by (b). Therefore, we have $b+0+0=b'+0+0$, hence $b+0=b'+0$.

(d) is an elementary calculation.
\end{proof}


\begin{remark}
    Of course, the mappings in {\rm (b)} and {\rm (c)} are not semigroup homomorphisms. 
\end{remark}


\smallskip

\section{Constructions of left \texorpdfstring{$RG$-}{-}semibraces}

\subsection{\texorpdfstring{Left $\mathbf{RG}$-}{-}semibraces from left averaging operators on groups}
In this subsection, we recall the notion of averaging operators on groups, as recently introduced in \cite{HuhuXing25}, and use it to provide a class of examples of left $RG$-semibraces. We underline that the study of averaging operators was initiated by Reynolds \cite{Rey} in connection with turbulence theory and experienced an algebraic renaissance in the work of Aguiar \cite{Aguiar00}.

\medskip

    A \emph{left averaging operator} on a group $B$ is a map ${\varphi}:B \to B$ such that the identity
    \begin{align*}
        {\varphi}(a){\varphi}(b) = {\varphi}\left( {\varphi}(a) b \right) 
    \end{align*}
    holds, for all $a,b \in B$.

\smallskip

Averaging operators on groups are a useful resource to construct right groups, see \cite[Proposition 3.8]{HuhuXing25}.
\begin{lemma}\label[lemma]{lemma:dashv}
    Let ${\varphi}:B \to B$ be a left averaging operator on a group $B$ and define a new operation $\vdash$ on $B$ by setting
    \begin{align*}
        a \vdash b := {\varphi}(a)b \,,
    \end{align*}
    for all $a,b \in B$.
   Then, $(B,\vdash)$ is a semigroup. Moreover, if there exists $e \in B$ such that $\varphi(e) = 1_B$, then $(B,\vdash)$ is a right group.
\end{lemma}

Let ${\varphi}:B \to B$ be a left averaging operator on a group $B$ and consider the operation $\vdash$ defined in \cref{lemma:dashv}.
If we set $\ker({\varphi}) := \{a \in B \mid {\varphi}(a) = 1_B \}$, then it is straightforward to verify that $\ker({\varphi}) = E(B,\vdash)$.
Accordingly, if we assume that $\ker(\varphi)$ is non-empty and fix $e \in \ker({\varphi})$ then each element $a \in B$ can be uniquely written as
\begin{align*}
    a = \left({\varphi}(a)e \right) \vdash \left( {\varphi}(a)^{-1}a \right) \,,
\end{align*}
where ${\varphi}(a)e$ and ${\varphi}(a)^{-1}a$ are the groupal and idempotent components of $a$, respectively, in relation to the decomposition of $(B,\vdash)$ induced by the fixed idempotent~$e$.

\begin{theorem}\label[theorem]{prop:avg_$RG$-semibrace}
    Let $B$ be a group, consider two left averaging operators ${\varphi},{\psi}:B \to B$ on $B$ and assume that $\ker(\psi) \neq \varnothing$.
    If $+$ and $\circ$ are the operations induced by ${\varphi}$ and ${\psi}$, respectively, then the structure $(B,+,\circ)$ is a left $RG$-semibrace if and only if
    \begin{align}\label{eq:averaging-RGsemibrace}
        {\psi}(a){\varphi}(b) = {\varphi}\left( {\psi}(a)b \right) {\varphi}(a)^{-1}{\psi}(a) 
    \end{align}
    holds, for all $a,b \in B$.
    \begin{proof}
       Assume first that \eqref{eq:averaging-RGsemibrace} holds and note that if $a \in \ker(\psi)$ then
        \begin{align*}
            \varphi(a) =
            \psi(a)\varphi(a) =
            \varphi\left( \psi(a)a \right) \varphi(a)^{-1}\psi(a) =
            \varphi(a)\varphi(a)^{-1}\psi(a) = 1_B \,.
        \end{align*}
        It follows that $\ker(\varphi) \neq \varnothing$ so that the pairs $(B,+)$ and $(B,\circ)$ are right groups.
        Now, let us note that if $a,b \in B$ then $-g_a+b$ coincides with the unique $x \in B$ such that $a+x = b$, from which follows that $-g_a+b = {\varphi}(a)^{-1}b$.
        With this observation, 
        it is straightforward to verify that \eqref{eq:averaging-RGsemibrace} implies \eqref{bjip1}.
        Conversely, if we assume that $(B,+,\circ)$ is a left $RG$-semibrace, then an analogous reasoning shows that \eqref{eq:averaging-RGsemibrace} is equivalent to \eqref{bjip1}, proving the assertion.
    \end{proof}
\end{theorem}

\smallskip

\begin{Ex}\label[example]{ex:endomorphism}
   Let $B$ be a group, ${\varphi}$ and ${\psi}$ two idempotent endomorphisms of $B$ such that ${\varphi}{\psi} = {\varphi}$ and $\Imm({\varphi})\subseteq Z(B)$. Then, by \cref{prop:avg_$RG$-semibrace}, the structure $(B,+,\circ)$ is a left $RG$-semibrace, with operations $+$ and $\circ$ defined by
    $
    a+b:= {\varphi}(a)b$
    and 
    $a\circ b:= {\psi}(a)b,
    $
    for all $a,b\in B$.\\
    A concrete non-trivial example is given by the Klein group $B:= \{1,\,x,\,y,\,xy\}$ and ${\varphi},{\psi}:B\to B$ the endomorphisms of $B$ obtained by setting ${\varphi}(x) = 1$, ${\varphi}(y) = y$,  and ${\psi}(x) = 1$, ${\psi}(y)= xy$,   respectively. Note that, in this case, 
    {the operations\  $+$\ and\ $\circ$\ do not coincide} since $y + x = xy$ and $y\circ x= y$. Note also that $G = \Imm({\varphi}) = \{1, y\}$, $H= \Imm({\psi}) = \{1,xy\}$, and $E = \ker({\varphi}) = \ker({\psi}) = F$.
\end{Ex}

\smallskip

\begin{remark}\label[remark]{rem:HG}
{\rm In the previous example, neither $G$ nor $H$ forms a skew sub-brace of $B$. Indeed, $G$ is not closed under the operation $\circ$ since $y\circ y = x\notin G$. 
    Moreover, $H$ is not closed under the operation $+$ since $xy + xy = x\notin H$. }
\end{remark}

\smallskip

\begin{Ex}\label[Ex]{ex:averaging}
    Let $B = \{ 1,\,x,\,y,\,z,\,xy,\,xz,\,yz,\,xyz \}$ be the group of order $8$ with exponent~$2$ and define two endomorphisms $\varphi,{\psi}:B \to B$ by setting
    \begin{align*}
        \begin{cases}
            \ {\varphi}(x) = 1 = {\varphi}(y) \\[0.2cm]
            \ {\varphi}(z) = z
        \end{cases}
        \quad\text{and}\qquad
        \begin{cases}
            \ {\psi}(x) = y = {\psi}(y) \\[0.2cm]
            \ {\psi}(z) = xyz \,
        \end{cases} 
    \end{align*}
    Then, ${\varphi}$, ${\psi}$ are idempotent endomorphisms such that $\varphi\psi = {\varphi}$, $\Imm({\varphi}) \subseteq Z(B)$ and $\ker({\varphi}) = \left\langle x,y \right\rangle \neq \left\langle xy \right\rangle = \ker({\psi})$.
    Therefore, following \cref{ex:endomorphism}, if $(B,+,\circ)$ is the left $RG$-semibrace whose operations are defined by setting
        $a+b := {\varphi}(a)b$ and
        $a\circ b := {\psi}(a)b$,
    for all $a,b \in B$, then we have that $F = \ker(\psi) \subsetneqq \ker(\varphi) = E$.
\end{Ex}

A straightforward \texttt{GAP} routine verifies that $8$ is the least positive integer witnessing an example such as the one above, i.e., where $F \subsetneq E$.

\smallskip

\begin{Ex}
    Let $B := \left\langle x,y \mid x^4 = y^2 = 1,\, xy = yx \right\rangle$ be the group of order $8$ isomorphic to $C_4 \times C_2$ and consider the functions $\varphi,\psi:B \to B$ defined by setting
    \begin{align*}
        \begin{cases}
        \ \varphi(1) = \varphi(x) = \varphi(x^2y) = \varphi(x^3y) = 1 \\[0.2cm]
        \ \varphi(y) = \varphi(x^2) = \varphi(xy) = \varphi(x^3) = x^2
        \end{cases}
        \,\text{and}\quad
        \begin{cases}
        \ \psi(1) = \psi(x) = 1 \\[0.05cm]
        \ \psi(x^2) = \psi(x^3) = x^2 \\[0.05cm]
        \ \psi(y) = \psi(xy) = y \\[0.05cm]
        \ \psi(x^2y) = \psi(x^3y) = x^2y 
        \end{cases}.
    \end{align*}
    Then $\varphi$ and $\psi$ are left averaging operators (that are not endomorphisms of $B$) satisfying the hypothesis of \cref{prop:avg_$RG$-semibrace} and therefore determine a left $RG$-semibrace structure $(B,+,\circ)$ whose operations are defined by setting $a + b = \varphi(a)b$ and $a \circ b = \psi(a)b$, for all $a,b \in B$.
    Setting $0$ equal to the identity element of the group $B$, it is straightforward to verify that $G = B + 0 = \Imm(\varphi) = \{0, x^2\}$, while $H = B \circ 0 = \Imm(\psi) = \{0,x^2,y,x^2y\}$ so that $G \cap H = G$.
    Thanks to \cref{prop_lambda_g_0}, we deduce that $(G,+,\circ)$ is a trivial skew left brace.
\end{Ex}

A quick \texttt{GAP} routine verifies that all left $RG$-semibraces of size $8$ arising from left averaging operators as in \cref{prop:avg_$RG$-semibrace} and admitting a decomposition $B = G + E = H \circ F$ such that $G$ is closed under $\circ$ (or equivalently such that $G \subseteq H$) all satisfy the identity $|G| = 2$, independently from the chosen multiplicative idempotent.

\medskip

\subsection{Matched product of left \texorpdfstring{$\mathbf{RG}$}-{-}semibraces}

This section aims to extend the construction of the matched product of left semibraces, introduced in \cite{CCSt20-1}, to the broader class of left $RG$-semibraces. As a particular case, we also focus on the semidirect product of left $RG$-semibraces.

\begin{lemma}\label[lemma]{lem-semidirect-rightgroup}
Let $(B, \cdot), (C, \cdot')$ be two right groups.  For every $u\in C$, let $\sigma_u:B\to B, a\mapsto\sigma_u(a)$ be a semigroup homomorphism. Suppose that~$\sigma\colon (C,\cdot')\to\Aut(B,\cdot),\, u \mapsto \sigma_u$ is a semigroup homomorphism from $(C, \cdot')$ into $\Aut(B, \cdot)$. Then $B\times C$ endowed with the following operation
\begin{align*}
(a,u) (b,v) := \left(a\cdot\sigma_u(b), \, u\cdot'v \right),
\end{align*}
for all $(a,u), (b,v)\in B\times C$, is a right group. \end{lemma}
\begin{proof}
It is a routine computation to verify that it is a semigroup. Hence, using  \cref{1.1}(b), it is sufficient to note that, if $(a,u),(b,v)\in B\times C$ and $x\in B$ and $y\in C$ are the unique elements such that $a\cdot x = b$ and $u\cdot' y = v$, set $t:= \sigma^{-1}_u(x)$, we obtain that $(t, y)$ is the unique element in $B\times C$ such that $(a,u)(t,y) = (b,v)$.
\end{proof}

Assuming that $B=G_B \cdot E_B$ and $C=G_C \cdot' E_C$ are the decompositions of the two right groups $(B, \cdot)$ and $(C, \cdot')$, respectively, as in \cref{1.1}(d), it follows from the definition that $\sigma_e=\id_B$, for all $e\in E_C$. As a consequence, we have that the idempotent part is given by $E_{B \times C}=E_B \times E_C$.  
Moreover, the groupal part $G_{B \times C}$ is computed in the following result.

\begin{lemma}\label[lemma]{lem:groupal_semidir}
    If $B \rtimes_\sigma C$ is the semidirect product of two right groups $(B,\cdot)$ and $(C,\cdot')$ via $\sigma$, then $G_{B \rtimes C} = G_B \rtimes_\sigma G_C$.
    In particular, $g_{(a,u)} = (g_a, g_u)$, for all $(a,u) \in B \times C$.
\end{lemma}
\begin{proof}
    The first part of the statement is obvious. Besides, we have that
    $$
    g_{(a,u)} 
    = (a,u)(0_B,0_C) 
    = (g_a \sigma_{g_u}(0_B),g_u),
    $$
    for all $(a,u) \in B \times C$. Note that the inverse of $g_{(a,u)}$ is
    \begin{align*}
        \left(\sigma_{g_u}^{-1}(g_a^{-}), g_u^{-} \right)
        =
        \left( ((\sigma^{-1}_{g_u}(g_a^{-})\cdot 0_B)\cdot\sigma_{g^{-1}_u}(0_B),\, g^{-1}_u \right) \in G_{B \rtimes_\sigma C} \,.
    \end{align*}
    Since
    \begin{align*}
        (0_B,0_C) =
        \left( ((\sigma^{-1}_{g_u}(g_a^{-})\cdot 0_B)\cdot\sigma_{g^{-1}_u}(0_B), g^{-1}_u \right)(g_a\cdot\sigma_{g_u}(0_B), g_u)
        = (\sigma^{-1}_{g_u}(0_B),0_C)
    \end{align*}
    we get that $\sigma^{-1}_{g_u}(0_B)  = 0_B$. 
    Therefore, the claim follows.
\end{proof}

\smallskip

\begin{theorem}\label{semidirect}
   Let $(B, +, \circ)$ and $(C, +', \circ')$ be two left $RG$-semibraces. Let $\sigma: (C, \circ') \to \Aut(B, +, \circ), \, u \mapsto \sigma_u$ be a semigroup homomorphism.
   Then $B\times C$ endowed with the following operations
\begin{align*}
&(a, u) \, \hat{+} (b, v) := (a+b, u+v),\\
&(a,u) \, \hat{\circ}\, (b,v) := \left(a\circ\sigma_u(b), \, u\circ'v \right),
\end{align*}
for all $(a,u), (b,v)\in B\times C$, is a left $RG$-semibrace which we call \emph{semidirect product of $B$ and $C$ via $\sigma$
} and denote it by $B \rtimes_\sigma C$.
\end{theorem}
\begin{proof}
    Clearly, $(B\times C, \, \hat{+})$ and $(B\times C, \, \hat{\circ})$ are right groups by the previous lemma. 
    Moreover, since by \cref{lem:groupal_semidir}, $g_{(a,u)} = (g_a, g_u)$, for all $(a,u)\in B\times C$, we obtain
    \begin{align*}
        (a,u)\hat{\circ}\left((b,v)\hat{+}(c,w)\right)
        &= (a\circ\sigma_u(b+c),\, u\circ'(v+w))\\
        &= (a\circ\sigma_u(b) - g_a + a\circ\sigma_u(c)),\,
        u\circ' v -' g_u +' u \circ' w))\\
        &= (a\circ\sigma_u(b),u\circ' v) \, \hat{-} \, (g_a, g_u) \, \hat{+}   \,(a\circ\sigma_u(c),u\circ' w)\\
        &= (a,u)\,\hat{\circ}\,(b,v) \hat{-} (g_a, g_u) \,
        \hat{+} \, (a,u) \, \hat{\circ} \, (c,w),
    \end{align*}
    for all $(a,u),(b,v),(c,w)\in B\times C$.
    Therefore, the claim follows.
\end{proof}

\smallskip

Thanks to \cref{semidirect}, it is possible to construct left $RG$-semibraces starting from any trivial left $RG$-semibrace and a specific (cancellative) semibrace. 

\begin{corollary}\label[corollary]{corollary:semidirect}
   Let $(B, \circ)$ be a group, 
$\left(C, +'\right)$ a right group,
and consider a semigroup homomorphism $\sigma: \left(C, +'\right) \to \Aut(B, \circ), \, u \mapsto \sigma_u$. Then $B \rtimes_\sigma C$ is a left $RG$-semibrace that is not a (cancellative) semibrace.
\end{corollary}
\begin{proof}
 The result is a direct application of \cref{semidirect}, by considering the trivial $RG$-semibrace $\left(C,+',+'\right)$ and the (cancellative) semibrace $(B,+,\circ)$ with $(B,+)$ a right zero semigroup, for which the identity $\Aut(B, +, \circ)= \Aut(B, \circ)$ holds.
\end{proof}

\smallskip

By applying \cref{corollary:semidirect} with a proper choice of $\sigma$ we provide a concrete example of a non-trivial $RG$-semibrace that is not a (cancellative) semibrace.
\begin{Ex}
Let $B:= \{1,\,x,\,y,\,xy\}$ be the Klein group. 
Define $a\circ b:= ab$, for all $a,b\in B$, and 
$a+'b:= \varphi(a)b$, for all $a, b \in B$, where 
${\varphi}:B\to B$ is the endomorphism of $B$ obtained by setting ${\varphi}(x) = 1$ and ${\varphi}(y) = y$ (as in \cref{ex:endomorphism}). Let $\sigma:\left(B, +'\right) \to \Aut(B,\circ)$ be the semigroup homomorphism from $(B,+')$ to $\Aut(B,\circ)$ given by $\sigma(x):= \id_B$ and $\sigma(y):= f$, where $f:B\to B$ is the automorphism of $(B,\circ)$ given by $f(x)=y$ and $f(y)=x$. Hence, $B\rtimes_{\sigma} B$ is a left $RG$-semibrace of order $16$.
\end{Ex}

\medskip
  
More generally, we can provide a construction of right groups that extends \cref{lem-semidirect-rightgroup}.
\begin{lemma}\label[lemma]{lemma_rg}
Let $(B,\cdot)$ and $(C,\cdot')$ be two right groups, $\sigma:C\to \Sym_B$ and $\delta:B\to C^{C}$ maps, and set $\sigma_u(a):= \sigma\left(u\right)\left(a\right)$ and $\delta_a(u):= \delta\left(a\right)\left(u\right)$, for all $a\in B$ and $u\in C$. If the following conditions are satisfied
	\begin{align}\label{S1}
	&\sigma_{u}(a \cdot b)=\,\sigma_u(a) \cdot\sigma_{\delta_a(u)}b	
	&\sigma_{u \cdot' v}(a) = \, \sigma_u\sigma_v(a) \tag{Z1}\\
	\label{S2}&\delta_a(u \cdot'v) = \delta_{\sigma_v(a)}(u) \cdot' \delta_a(v) & \delta_{a \cdot b}(u) = \delta_b\delta_a(u) \tag{Z2}
	\end{align}
for all $a,b \in B$ and $u,v \in C$, then $B \times C$ is a right group with respect to the operation defined by
	\begin{equation}\label{prod-semigruppo-match}
	\left(a,u\right) \left(b,v\right) :=  
    \left(a \cdot \sigma_u(b), \delta_b(u) \cdot v\right),
	\end{equation}
	for all $a,b \in B$ and $u,v \in C$.
    \end{lemma}
    \begin{proof}
       It is a routine computation to verify that $B \times C$ with the operation \eqref{prod-semigruppo-match} is a semigroup. Now, let $(a,u),(b,v) \in B \times C$
       and consider $x\in B$ the unique element such that $a\cdot x = b$, obtained using \cref{1.1}(b).
       If we set $t:= \sigma^{-1}_u(x)$ and consider the unique $y\in C$ such that $\delta_t(u) \cdot' y=v$, again through \cref{1.1}(b), then we obtain that $(t, y)$ is the unique element in $B\times C$ such that $(a,u)(t,y) = (b,v)$. 
    \end{proof}

\smallskip

Observe that $B \times C$, endowed with the operation \eqref{prod-semigruppo-match}, is the classical \emph{Zappa--Szép (matched) product} of the two semigroups $(B, \cdot)$ and $(C, \cdot')$, cf. \cite{Ku83}. \\
Following the approach of matched products for left semibraces introduced in \cite{CCSt20-1}, we define appropriate maps $\alpha$ and $\beta$ in order to obtain a new left $RG$-semibrace whose multiplicative semigroup is isomorphic to the Zappa–Szép product of the given right groups.
   
    \smallskip

\begin{definition}
    {\rm Let $(B, +, \circ)$ and $(C, +', \circ')$ be two left $RG$-semibraces. Consider $\alpha:C\to\Aut(B)$ a semigroup homomorphism from the right group $(C,\circ')$ into the automorphism group of the semigroup $(B,+)$, and $\beta:B\to\Aut(C)$ a semigroup homomorphism from the right group $(B,\circ)$ into the automorphism group of the semigroup $(C,+')$. If $\alpha$ and $\beta$ satisfy the following conditions
    \begin{align}\label{eq:lambda-alpha}
    \alpha_u\lambda_a = \lambda_{\alpha_u(a)}\alpha_{\beta^{-1}_{\alpha_u(a)}(u)}
    \qquad\text{and}\qquad
        \beta_a\lambda_u = \lambda_{\beta_a(u)}\beta_{\alpha^{-1}_{\beta_a(u)}(a)},
    \end{align}
    for all $(a,u), (b,v)\in B\times C$, we name the quadruple $(B, C, \alpha, \beta)$ a \emph{matched product system of the left $RG$ semibraces $B$ and $C$ via $\alpha$ and $\beta$}.}
\end{definition}

\smallskip

\begin{theorem}
      Let $(B, +, \circ)$ and $(C, +', \circ')$ be two left $RG$-semibraces and $(B, C, \alpha, \beta)$ a matched product system of the left $RG$ semibraces $B$ and $C$ via $\alpha$ and $\beta$. 
      Then $B\times C$ endowed with the following operations
\begin{align*}
&(a, u) \, \hat{+} (b, v) := (a+b, u+'v)\,,\\
&(a,u) \, \hat{\circ}\, (b,v) :=
\left(\alpha_{u}{\left(\alpha^{-1}_{u}{\left(a\right)}\circ b\right)},\,\beta_{a}\left(\beta^{-1}_{a}\left(u\right)\circ' v \right)\right)\,,
\end{align*}
for all $(a,u), (b,v)\in B\times C$, is a left $RG$-semibrace which we call \emph{matched product of $B$ and $C$ via $\alpha$ and $\beta$} and denote by $B \bowtie_{\alpha,\beta} C$.
\end{theorem}
\begin{proof}
First,  by \eqref{eq:lambda-alpha}, we obtain that 
    \begin{align*}
        \alpha_u(\alpha^{-1}_u(a)\circ b) 
        &= \alpha_u\left(\alpha^{-1}_u(a) + \lambda_{\alpha^{-1}_a(u)}(b)\right) 
        = a + \alpha_u\lambda_{\beta^{-1}_a(u)}(b)\\ 
        &= a + \lambda_a\alpha_{\beta^{-1}_a(u)}(b)
        = a\circ\alpha_{\beta^{-1}_a(u)}(b),
    \end{align*}
     and, similarly,
    $\beta^{}_a\left(\beta^{-1}_a(u)\circ' v\right) 
     = u\circ'\beta^{}_{\alpha^{-1}_u(a)}(v),$
     for all $(a,u),(b,v) \in B \times C$. \\
Now, observe that the map $\varphi:B\times C\to B\bowtie_{\alpha,\beta}C$ defined by $\varphi(a,u)= (a,\beta_a(u)),$ for all $(a,u)\in B\times C$, is an isomorphism from the Zappa--Szép product semigroup $(B\times C, \circ)$ via $\sigma$ and $\delta$, where $\sigma:C\to \Sym_B$ and $\delta:B\to C^C$ are the maps given by
     $$
     \sigma_u:= \alpha_u
     \qquad\text{and}\qquad\delta_a(u):= \beta^{-1}_{\alpha_u(a)}(u),
     $$ 
     for all  $a\in B$ and $u\in C$, to the structure $(B\bowtie_{\alpha,\beta}C, \hat{\circ})$. Indeed, clearly $ \varphi$ is bijective and
     \begin{align*}
         \varphi((a,u)&\circ(b,v))
         = \left (a\circ\alpha_u(b),\, \beta_a\beta_{\alpha_u(b)}\left(\beta^{-1}_{\alpha_u(b)}(u)\circ' v\right) \right)\\
         &\underset{\eqref{eq:lambda-alpha}}{=}(a\circ\alpha_u(b),\, \beta_a(u\circ'\beta_b(v)))
         = (a\circ\alpha_u(b),\, \beta_a(u +' \lambda_u\beta_b(v)))\\
         &=(a\circ\alpha_u(b),\, \beta_a(u) +' \beta_a\lambda_u\beta_b(v))\\
         &\underset{\eqref{eq:lambda-alpha}}{=} \big(a\circ\alpha_u(b),\, \beta_a(u) +' \lambda_{\beta_a(u)}\beta_{\alpha^{-1}_{\beta_a(u)}(a)}\beta_b(v)\big)
         \\
         &= \big(a\circ\alpha_u(b),\, \beta_a(u)\circ'  \beta_{\alpha^{-1}_{\beta_a(u)}(a)}\beta_b(v)\big)= (a,\beta_a(u))\,\hat{\circ}\,(b,\beta_b(v))\\
         &= \varphi(a,u)\, \hat{\circ}\, \varphi(b,v),
     \end{align*}
     for all $(a,u),(b,v)\in B\times C$. Hence, by \cref{lemma_rg}, $\left(B \times C, \hat{\circ}\right)$ is a right group.
Furthermore, for all $(a,u),(b,v),(c,w)\in B\times C$,
     \begin{align*}
         (a, u)&\,\hat{\circ}\, \left((b,v) \,\hat{+}\, (c,w)\right)
         = \left( a\circ\alpha_{\beta^{-1}_a(u)}(b+c),\, 
         u\circ'\beta_{\alpha^{-1}_u(a)}(v+'w) \right)\\
         &= \left( a\circ\alpha_{\beta^{-1}_a(u)}(b) - g_a + a \circ \alpha_{\beta^{-1}_a(u)}(c),\, 
         u\circ'\beta_{\alpha^{-1}_u(a)}(v) -' g_u +' \beta_{\alpha^{-1}_u(a)}(w) \right)\\
         &=(a, u) \,\hat{\circ}\, (b,v) \,\hat{-}\, g_{(a,u)} \,\hat{+}\, (a,u) \,\hat{\circ}\, (c,w).
     \end{align*}
     Therefore, the claim follows.
\end{proof}

\medskip

\section{Solutions associated to left \texorpdfstring{$RG$-}{-}semibraces}

In this section, we prove that any left $RG$-semibrace gives rise to a left non-degenerate solution. 

\smallskip

To prove the main result, we make use of the notion of \emph{twist} of a shelf 
 and the description of left non-degenerate solutions in terms of twists 
 (cf. \cite[Definition 2.14]{DoRySte24}). 
First, recall that a \emph{{\rm(}left{\rm)} shelf} $(X, \triangleright)$ is a set $X$ equipped with a left self-distributive binary operation $\triangleright$ such that
\begin{equation}\label{self_distri}
    x \triangleright (y \triangleright z)= (x \triangleright y) \triangleright (x \triangleright z), 
\end{equation} 
 for all $x,y,z\in X$. 
A \emph{shelf homomorphism} between two shelves is defined as any map that preserves the shelf operations on each underlying set.  
If we denote with $L_x:X \to X, y \mapsto x \triangleright y$ the operator of left multiplication by $x$, for all $x \in X$, then equation \eqref{self_distri} is equivalent to requiring that $L_x$ be a shelf homomorphism. If, in addition, the maps $L_x$ are bijective, $(X, \triangleright)$ is a \emph{(left) rack}. Moreover, we say that a shelf $(X, \triangleright)$ is a \emph{(left) spindle} if $x \triangleright x=x$, for all $x\in X$. A rack $(X, \triangleright)$ that is also a spindle is a \emph{left quandle}. For more details on the topic, see, for instance, \cite{FeSaFa04}. 

   If $(X,\triangleright)$ is a shelf, then a map $\varphi:X \to \Aut(X,\triangleright)$ is called a \emph{twist} of $(X,\triangleright)$ if the following identity 
\begin{align}\label{eq:twist}
    \varphi_a\varphi_b = 
    \varphi_{\varphi_a(b)}\varphi_{\varphi^{-1}_{\varphi_a(b)}\left(\varphi_a(b)\, \triangleright\,a\right)},
\end{align}
holds, for all $a, b \in X$.

\begin{lemma}{\rm \cite[cf. Theorem 2.15]{DoRySte24}}\label[lemma]{thm:twist_solution}
Let $(X,\triangleright)$ be a shelf and $\varphi:X \to\Sym_X$ a map.
    If we define $r:X\times X\to X\times X$ by setting
    \begin{align*}
        r(a,b) =
        \left(\varphi_a(b), \varphi^{-1}_{\varphi_a(b)}\left(\varphi_a(b) \triangleright a\right)\right) \,,
    \end{align*}
    for all $a,b \in X$, then $(X,r)$ is a left non-degenerate solution if and only if $\varphi$ is a twist.
Moreover, any left non-degenerate solution can be obtained that way.
\end{lemma}

\smallskip

We now introduce a particular spindle that will be used, together with \cref{thm:twist_solution}, to prove the theorem on solutions. The proof of the following lemma is omitted, since it consists of a straightforward computation.
\begin{lemma}
   Let $(S,\,\cdot)$ be a right group and $\triangleright$ the binary operation on $S$ given by
    $$
    a\triangleright b:= g_a^{-1}\,b\,a,
    $$
    for all $a,b\in S$. Then, the structure $(S,\, \triangleright)$ is a spindle that we call \emph{conjugation spindle} on $(S,\,\cdot)$.
\end{lemma}

\smallskip

Given any left $RG$-semibrace $B$, we introduce and investigate, for every $b\in B$, a map $\rho_b$ 
that serves as the second component of the solutions arising from $B$, relating it to a twist of the conjugation spindle.
\begin{proposition}\label[proposition]{rho}
    Let $B$ be a left $RG$-semibrace and, for every $b\in B$, define a map $\rho_b:B\to B$ by setting 
    \begin{align}\label{def_rho}
        \rho_b(a) = h_{\lambda_a(b)}^-\circ a \circ b,
    \end{align}
    for all $a \in B$. Then, the following hold:
    \begin{itemize}
        \item[\rm (a)] $a \circ b=\lambda_a(b) \circ \rho_b(a)$, for all $a, b \in B$.
        \item[\rm (b)] $\rho\colon B\to B^B$, $b\mapsto\rho_b$, is a semigroup anti-homomorphism from $(B,\circ)$ to $B^B$.
       \item[\rm (c)] $\rho_b(B)\subseteq H\circ f_b$, for  every $b\in B$.
         \item[\rm (d)]  $\rho_b(h \circ f)=\rho_b(h)$, for all $b \in B$, $h \in H$, $f \in F$. 
    \end{itemize}
\end{proposition}
\begin{proof}
  Statement (a) is exactly \eqref{def_rho} by Theorem~\ref{1.1}(b). Moreover, if $a,b,c\in B$, by \cref{lambda1}, we have
    \begin{align*}
        \lambda_a(b) \circ \lambda_{\rho_b(a)}(c)&= \lambda_a(b)\circ \lambda_{h^-_{\lambda_a(b)}\circ a\circ b}(c)= \lambda_a(b)\circ 
        \lambda_{h^-_{\lambda_a(b)}}\lambda_a\lambda_b(c)\\
        &= \lambda_a(b)\circ h^-_{\lambda_a(b)}
        \circ\left(h_{\lambda_a(b)} + \lambda_a\lambda_b(c)\right)\\
        &= g_{h_{\lambda_a(b)}} + \lambda_a\lambda_b(c)
        \underset{\eqref{prop_GHEF}(a)}{=} g_{\lambda_a(b)} + e_{\lambda_a(b)} +\lambda_a\lambda_b(c)\\
        &= \lambda_a(b) + \lambda_a\lambda_b(c)= \lambda_a(b\circ c),
    \end{align*}
   where, in the last equality, we use Proposition \ref{lambda1}(c) and \cref{somma_cerchietto}(a). Thus, we obtain that
    \begin{align*}
        \rho_c\rho_b(a)
        &=  h^-_{\lambda_{\rho_b(a)}(c)} \circ h^-_{\lambda_a(b)} \circ a \circ b \circ c
        =  h^-_{\lambda_a(b) \circ \lambda_{\rho_b(a)}(c)} \circ a \circ b \circ c\\
        &= h^-_{\lambda_a(b\circ c)}\circ a\circ b\circ c
        = \rho_{b\circ c}(a),
    \end{align*}
    hence (b) holds. As far as (c) is concerned, assume that $b=h\circ f$ with $h\in H$ and $f\in F$, so that $f=f_b$. From Definition~(\ref{def_rho}) and (b), for all $a \in B$, we have $$
    \rho_b(a)
    = \rho_f\rho_h(a)
    = h^-_{\lambda_{\rho_h(a)}(f)} \circ h^-_{\lambda_h(a)} \circ h \circ h_a\circ f \in H \circ f.
    $$  
    Moreover, again, from Definition~(\ref{def_rho}), $\lambda_{f}=\id_B$ implies that 
    $$\rho_b\left(h\circ f \right)=h^-_{\lambda_{h \circ f}(b)}\circ h \circ f \circ b=h^-_{\lambda_{h}(b)}\circ h \circ b=\rho_b(h).$$
    This proves (d).
\end{proof}

Notice that $\rho_b(a)$ is $ (\lambda_a(b))\backslash (a \circ b)$, where $\backslash$ is the inverse of the operation $\circ$ of the right group $(B,\circ)$.

In the following, we gather several properties of the map $\rho$. Before stating the result, we introduce the mappings $\mu_f:H \to B$, for all $f \in F$, and $\pi_H:B \to H$ defined, respectively, by setting 
\begin{align*}
    \mu_f(x) = x \circ f \qquad\text{and}\qquad \pi_H(b) = b \circ 0 \,,
\end{align*}
for all $x \in H$ and $b \in B$. Also, for every mapping $f\colon X\to X$ we will denote by $\sim_f$ the equivalence relation on the set $X$ defined by $x\sim_f y$ if $f(x)=f(y)$.

\begin{proposition}\label{daria}
    Let $B$ be a left $RG$-semibrace. For all  $b\in B$, let $\rho'_b\colon H\to H$ be the map defined by
         $$\rho'_b(h)=\rho_b(h) \circ 0$$
         for all $h \in H$.
         Then the following hold: 
    \begin{itemize}
         \item[\rm (a)]
         $\rho_b = \mu_{f_b}\rho'_b\pi_H$\, and\,
         the mapping $\rho'\colon B\to H^H$, $b\mapsto \rho'_b$, is a semigroup anti-homomorphism of $(B,\circ)$ into $H^H$.
  \item[\rm (b)] { $\rho_0\colon B\to B$ and $\varepsilon:=\rho'_0\colon H\to H$ are idempotent mappings with the same image $A$, where $A=(G+F)\cap H$.
   \item[\rm (c)] $\varepsilon H^H\varepsilon$ is a subsemigroup of $H^H$, $\varepsilon H^H\varepsilon$ is a monoid with identity $\varepsilon$, and $\rho'(H)\subseteq U\left(\varepsilon H^H\varepsilon\right)$, the group of units of $\varepsilon H^H\varepsilon$.
  \item[\rm (d)] 
  An element $f\in H^H$ belongs to $\varepsilon H^H\varepsilon$ if and only if $f(H)\subseteq A$ and $\sim_\varepsilon\subseteq\sim_f$.
   \item[\rm (e)] 
   An element $f\in H^H$ belongs to $U\left(\varepsilon H^H\varepsilon\right)$ if and only if $f(H)= A$ and $\sim_\varepsilon=\sim_f$.
   \item[\rm (f)] $U(\varepsilon H^H\varepsilon)\cong\Sym_A$, via the isomorphism $\Phi$ that maps every $f\in U(\varepsilon H^H\varepsilon)$, $f\colon H\to H$, to its restriction 
   $f_{|_A}\colon A\to A$.
    \item[\rm (g)] The restriction $\rho'_{|_H}\colon H\to U\left(\varepsilon H^H\varepsilon\right)$, $h\mapsto \rho'_h$, is a group anti-homomorphism of $(H,\circ)$ into $U\left(\varepsilon H^H\varepsilon\right)$.
  \item[\rm (h)] Let $\varphi\colon B\to H\times F$, $x \mapsto(h_x,f_x)$, denote the canonical isomorphism (relative to the choice of the element $0$ in $F$). Then, for all  $b\in B$, 
  \begin{align*}
  \rho_b
  = \mu_{f_b}\rho'_b\pi_H
  =\varphi^{-1}(\rho'_{h_b}\times c_{f_b})\varphi,
  \end{align*}
  where $\rho'_{h_b}\times c_{f_b}\colon H\times F\to H\times F$, $\rho'_{h_b}\colon H\to H$ belongs to $U(\varepsilon H^H\varepsilon)$, and $c_{f_b}\colon F\to F$ is the mapping constantly equal to $f_b$. In particular, 
  $$
  \rho_h=\varphi^{-1}(\rho'_{h}\times c_0)\varphi
  \qquad\text{and}\qquad
  \rho_f=\varphi^{-1}(\varepsilon\times c_f)\varphi,
  $$
  for all $h\in H$ and $f\in F$.
 \item[\rm (i)] The semigroup homomorphism $\rho\colon B^{\op}\to B^B$ of the  left group $B^{\op}$ into the monoid $B^B$ factors through the left group $F^{\op}\times \Sym_A$.}
\end{itemize}\end{proposition}
    \begin{proof}
    (a) 
    If $a,b \in B$, then
    \begin{align*}
        \mu_{f_b}\rho'_b\pi_H(a) &=
        \rho'_b(a \circ 0) \circ f_b =
        \rho_b(a \circ 0) \circ 0 \circ f_b =
        \rho_b(a) \,,
    \end{align*}
    where the last equality is a consequence of (c) and (d) in \cref{rho}.
      In addition, the mapping\, 
      $\rho'$\, 
      is a semigroup anti-homomorphism of $(B,\circ)$ into the monoid $H^H$ by \cref{rho}(b).
 
    (b)\,--\,(c) Since $\rho'\colon B\to H^H$ is a semigroup homomorphism of $B^{\op}$ into $H^H$, its restriction $\rho'_{|_H}\colon H\to H^H$ is a semigroup homomorphism of the semigroup $H^{\op}$ into $H^H$. Now $0$ is an idempotent of $H$, so that its image $\varepsilon:=\rho'_0$ via the semigroup homomorphism $\rho'_{|_H}$ is an idempotent of $H^H$, and the identity $0\circ h\circ 0=h$, for every $h\in H$, implies that the image $\rho'_{|_H}(H)$ of $\rho'_{|_H}$ is contained in $\varepsilon H^H\varepsilon$. Statement (c) now follows immediately. Similarly, $\rho\colon B\to B^B$ is a semigroup homomorphism of $B^{\op}$ into $B^B$, so that $\rho_0\colon B\to B$ is idempotent. From (a), we get that $\rho_0=\mu_{f_0} \rho'_b\pi_H$, where $\mu_{f_0} $ is the embedding of $H$ in $B$ and $\pi_H$ is an onto mapping. It follows that $\rho$ and $\rho'$ have the same image $A$, say.
    
  By \cref{prop_GHEF}(a), $\lambda_b(0)=e_{h_b}$ for every $b\in B$, so that $$\rho_0(b)
        = h_{\lambda_b(0)}^-\circ b \circ 0=h_{e_{h_b}}^-\circ b \circ 0.$$ Hence, $\rho'_0(h)=\rho_0(h) \circ 0=h_{e_{h}}^-\circ h$, for every $h\in H$.
        For any idempotent mapping, the image is equal to the set of its fixed points, so 
        \begin{align*}
            A&=\{\,h\in H\mid \rho'_0(h)=h\,\}=\{\,h\in H\mid h_{e_{h}}^-\circ h=h\,\}=\{\,h\in H\mid h_{e_{h}}^-=0\,\}\\
            &=\{\,h\in H\mid e_h\in F\,\}=\{\,h\in H\mid \pi_E(h)\in F\,\}=
       \left( (\pi_E)^{-1}(F )\right)\cap H\\
       &=(G+F)\cap H.
        \end{align*} Here $\pi_E\colon B\to B$ is the mapping defined by $\pi_E(b)=e_b=b\backslash b$, where $\backslash$ is the inverse operation of the operation $+$ of the right group $B.$
       
       (d), (e), (f) and (g)  are elementary exercises. For instance, if $f\in \varepsilon H^H\varepsilon$, then $f\varepsilon=f$, so that $\sim_\varepsilon\subseteq\sim_f$. Similarly $\varepsilon f=f$, which implies $f(H)\subseteq\varepsilon(H)=A$.
       
       (h)  For every $h\in H,\ f\in F$, we have that $$\rho_b(h\circ f)=\mu_{f_b} \rho'_{h_b}\pi_H(h\circ f)=\mu_{f_b} \rho'_{h_b}(h)=\rho'_{h_b}(h)\circ f_b$$ and $$\varphi^{-1}\left(\rho'_{h_b}\times c_{f_b})\varphi(h\circ f\right)=\varphi^{-1}\left(\rho'_{h_b}\times c_{f_b}\right)(h,f)= \varphi^{-1}\left(\rho'_{h_b}(h),f_b\right)=\rho'_{h_b}(h)\circ f_b.$$ Therefore, $\rho_b=\varphi^{-1}\left(\rho'_{h_b}\times c_{f_b}\right)\varphi$. For $b=h\in H$, we get $\rho_h=\varphi^{-1}\left(\rho'_{h}\times c_0\varphi\right)$, while for $b=f\in F$, we find that $\rho_f=\varphi^{-1}\left(\rho'_0\times c_{f}\right)\varphi=\varphi^{-1}\left(\varepsilon\times c_{f}\right)\varphi$.
           
           (i) The semigroup homomorphism $\rho\colon B^{\op}\to B^B$ is the composite mapping of:
           (1) the semigroup isomorphism $B^{\op}\to F^{\op}\times H^{\op}$,\\
           (2) the semigroup morphism $F^{\op}\times H^{\op}\to F^F\times U\left(\varepsilon H^H\varepsilon\right)$, $(f,h)\mapsto \left(c_f,\rho'_h\right)$,\\  
           (3) the semigroup isomorphism $F^F\times U\left(\varepsilon H^H\varepsilon\right)\to F^F\times \Sym_A$, \\
           (4) the semigroup homomorphism $F^F\times \Sym_A\to B^B$ that maps every element $(g, \sigma)\in F^F\times \Sym_A$ to the mapping $\varphi^{-1}\left(g,\Phi^{-1}(\sigma)\right)\varphi\colon B\to B$.
 \end{proof}

\smallskip

\begin{theorem}\label[theorem]{t}
    Let $B$ be a left $RG$-semibrace. Then the map~\hbox{$r_B:B\times B\to B\times B$} defined by~$r_B(a,b)=\left(\lambda_a(b), \rho_b(a)\right),$
    for all $a, b \in B$, is a left non-degenerate solution. 
    \begin{proof}
 Let us prove that the map $\lambda$ is a twist of the conjugation spindle $(B, \triangleright)$ where $a\triangleright b := -g_a + b + a$, for all $a,b\in B$. 
By \cref{lambda1}(a),
\begin{align*}
    g_{\lambda_x(a)}  
    = \lambda_x(a) + 0 
    = \lambda_x(g_a) + \lambda_x(e_a) +  0 
    = \lambda_x(g_a) +  0,
\end{align*}
for all $x,a\in B$. Hence, we obtain that
\begin{align*}
    \lambda_x(a\triangleright b) 
    = \lambda_x(-g_a) + \lambda_x(b) + \lambda_x(a)
    = -g_{\lambda_x(a)} + \lambda_x(b)+ \lambda_x(a)
    = \lambda_x(a)\triangleright\lambda_x(b)
\end{align*}
for all $x,a,b\in B$, namely, for all $x\in B$, $\lambda_x$ is a shelf homomorphism and since it is bijective, $\lambda_x\in \operatorname{Aut}(B, \triangleright)$. 
Now, let us observe that
\begin{align*}
  \lambda_x \rho_{\lambda^{-1}_y(x)}(y)
    &= \lambda_x\left(h^-_{\lambda_y  \lambda_{h^-_y}(x)}\circ y\circ \lambda_{h^-_y}(x)\right)
    = \lambda_x\left(h^-_{x}\circ (h_y + x)\right)\\
    &= -g_x +  x\circ h^-_x\circ (h_y + x)
    = -g_x + h_y  + x =  -g_x + g_{h_y}  + x\\
   & = -g_x + g_y  + x = x\triangleright y,
\end{align*}
for all $x,y\in B$, where in the second last equality  we 
use \cref{prop_GHEF}(a).
Therefore, by the bijectivity of the maps $\lambda_x$, we obtain that 
$$
\rho_y(x) = \lambda^{-1}_{\lambda_x(y)}\left(\lambda_x(y)\triangleright x\right),
$$ 
for all $x,y\in B$. As a consequence, 
equality \eqref{eq:twist} follows by \cref{lambda1}(c) and \cref{rho}(a). 
Therefore, $r_B$ is a left non-degenerate solution by \cref{thm:twist_solution}. 
    \end{proof}
\end{theorem}

\smallskip

From Remark~\ref{1e2} and Proposition~\ref{daria}(h), we get the following result.

\begin{corollary}\label[corollary]{cor:GH_decomp}
 Let $B$ be a left $RG$-semibrace. Then the corresponding left non-degenerate solution $r_B\colon B\times B\to B\times B$ can be identified with the mapping $$r_B\colon (G\times E)\times (H\times F)\to (G\times E)\times (H\times F)$$ defined by $$\left(\left(g_a,e_a\right),\left(h_b,f_b\right)\right)\mapsto \left(\left(\left(\lambda_{g_a+e_a}\right)_1(g^{}_{h_b \circ f_b}),e_{(g_a+e_a) \circ (h_b \circ f_b)}\right), \left(\rho'_{h_b \circ f_b}(h_{g_a+e_a}), f_b \right)\right).$$ 
\end{corollary}
\begin{proof}
   For all $a,b\in B$, we have
    \begin{align*}
\lambda_{a}\left(b\right)&= -g_a+a \circ b + e_{a \circ b}
=-g_a+a \circ \left(g_b+e_b\right)+e_{a \circ b}\\
&=-g_a+a \circ g_b+\lambda_a\left(e_b\right)+e_{a \circ b}=\lambda_a\left(g_b\right)+e_{a \circ b},
    \end{align*}
  where in the last equality we 
apply that $\lambda_a\left(e_b\right) \in E$.  
  Hence, following the notation in  Remark~\ref{1e2}, $g_{\lambda_a(b)}=\left(\lambda_{a}\right)_1(b)$ and $e_{\lambda_a(b)}=e_{a \circ b}$. Moreover, by \cref{lambda1}(d),
\begin{align*}
\rho^{}_{b}(a)
=  h^-_{\lambda_{h_a}(b)}\circ a\circ b 
= h^-_{\lambda_{h_a}(b)}\circ h_a\circ b\circ f_b
= \rho^{}_{b}(h_a)\circ f_b.
\end{align*}
Thus, $h_{\rho_b(a)}=\rho'_b\left(h_a\right)$ and $f_{\rho_b(a)}=f_b$.
Therefore, the claim follows.
\end{proof}

\smallskip

\begin{Ex}\label[example]{es_banale}
    Let $(B, \cdot)$ be a right group 
    that is not a group and consider the trivial left $RG$-semibrace $(B, \cdot, \cdot)$. Then the map $r_B(x, y)=\left(y, \, h^{-}_{y} \cdot h_x \cdot y\right)$ is a left non-degenerate solution on $B$ that cannot arise as a solution associated with a (cancellative) semibrace on $B$. 
\end{Ex}

\smallskip

Let us now focus on the order of solutions associated to a left $RG$-semibrace. To this purpose, we recall the notions of \emph{index} and \emph{period} of a solution $r$ introduced in \cite{CCSt20-2} that are respectively defined as 
\begin{align*}
    &\indd{\left(r\right)}
    :=\min\left\{\left.j \,\right|\, j\in\mathbb{N}_0, \, \exists \, l\in \mathbb{N}\ r^j=r^l , \ j\neq l\right\},\\
    &\perr{\left(r\right)}
    :=\min\left\{\left.k\,\right| \, k\in\mathbb{N}, \, r^{\indd{\left(r\right)}+k}=r^{\indd{\left(r\right)}}\right\}.
\end{align*}
While being slightly different from the classical ones (cf.~\cite[p.~10]{Ho95}), these definitions are functional to distinguish bijective solutions, having index $0$, from non-bijective ones, having index a positive integer.  

\begin{proposition}
    If $B$ is a left $RG$-semibrace and $r_B$ is the solution associated with $B$, then the following hold:
    \begin{itemize}
        \item[{\rm(1)}] $r_B$ has finite index and period if and only if $G/Z(G)$ is periodic.
        \item[{\rm(2)}] If $n := \exp(G/Z(G))$ is finite, then $r_B^{2n + 1} = r_B$.
    \end{itemize}
    In particular, $r_B = 2n$ if and only if $|E| = 1$.
    \end{proposition}
\begin{proof}
    By \cref{thm:twist_solution} and \cite[Lemma 2.12]{DoRySte24}, it is enough to prove the statement for the solution associated to the conjugation spindle on the additive right group $(B,+)$, namely, to the map $r_\triangleright:B\times B\to B\times B$ defined by $r_\triangleright(a,b) = (b,\, - g_b + a+b)$, for all $a,b\in B$.
    To this end, note that $r_\triangleright$ is the solution associated to the trivial $RG$-semibrace $(B,+,+)$, hence, by \cref{cor:GH_decomp}, we can identify the map $r_\triangleright$ with
    \begin{align*}
        r_\triangleright:\,& (G \times E) \times (G \times E) \to (G \times E) \times (G \times E), \\
        &\left( (g_a,e_a), (g_b,e_b) \right)\mapsto  \left( (g_b,e_b), (g_b \blacktriangleright g_a, e_b) \right)
    \end{align*}
    where $(G,\blacktriangleright)$ is the conjugation quandle of the additive group $(G,+)$.
    As a consequence, the solution $r_\triangleright$ is the direct product $r_{\blacktriangleright} \times r_E$ where $r_\blacktriangleright$ is the solution associated with $(G,\blacktriangleright)$ and $r_E$ is the idempotent solution associated to the $RG$-semibrace $(E,+,+)$.
    Thanks to \cite[Theorem 9]{CCSt20-2}, we deduce that $r_\triangleright$ has finite order if and only if $r_G$ has finite order.
    Moreover, by \cite[Corollary 1.20]{AlSte25}, $r_G$ has finite order if and only if $G/Z(G)$ is periodic.
    In particular, by \cite[Proposition 10]{CCSt20-2} in this case we have that $\indd{\left(r_\triangleright\right)}\in \{0,1\}$ and $\perr{\left(r_\triangleright\right)} = 2n$, where $n := \exp(G/Z(G))$. Therefore, the claim follows.
\end{proof}

\smallskip

\begin{remark}
Note that solutions obtained from left $RG$-semibraces 
have the same index and same order as those arising from left (cancellative) semibraces, cf. \cite[Corollary 13]{CCSt20-2}.
\end{remark}

\smallskip
Let us notice that index of $r_B$ depends exactly by the substructure $(E,+,\circ)$ of the arbitrary $RG$-semibrace $(B,+,\circ)$. On the other hand, although by \cref{prop_lambda_g_0} the structure $(G,+,\circ)$ is not generally a substructure of a left $RG$-semibrace $B$, the period of $r_B$ depends exactly on the additive group $(G,+)$. 
More specifically, the following result aligns with \cite[Corollary 13]{CCSt20-2} and \cite[Theorem 5.1]{JeAr19}.

\begin{corollary}\label[corollary]{prop-id-sol}
 Let $B$ be a left $RG$-semibrace and $r_B$ the solution associated with~$B$. Then $r_B$ is cubic, i.e., $r_B^3  = r_B$, if and only if $(G,+)$ is an abelian group. In particular, $r_B$ is idempotent if and only if $G=\{0\}$, namely, $(B,+) = (E,+)$  is a left zero semigroup.
\end{corollary}

\smallskip

\begin{Ex}
    Let $(B,\circ)$ be a right group and define $a+b:= b$, for all $a,b\in B$, i.e., $(B,+)$ is a right zero semigroup. Then, $(B,+,\circ)$ is an $RG$-semibrace and the solution associated to $B$ is the map $r_B:B\times B\to B\times B$ given by
    \begin{align*}
        r_B(a,b) = (a\circ b,\, f_b),
    \end{align*}
    for all $a,b\in B$. Consistently with \cref{prop-id-sol}, $r_B$ is idempotent.
\end{Ex}

\smallskip

As for \cref{es_banale}, the solution arising from the following example of left $RG$-semibrace does not arise from a left (cancellative) semibrace.
\begin{Ex}
Let $B$ the left $RG$-semibrace in \cref{ex:endomorphism}. Observe that the solution $r_B$ cannot be obtained from any left (cancellative) semibrace structure $\left(B, +', \circ'\right)$ on the same underlying set $B$. Indeed, since $F=\{1,x\}$, we have that $\lambda_1(1)=1$ and $\lambda_x(x)=x$. Hence, the only possible candidates for the identity element $0$ of the group $\left(B, \circ'\right)$ are $1$ and $x$ (this is a consequence of the fact that in any left (cancellative) semibrace $\lambda_0(0)=0$). Assume that $0$ is $1$, then considering that, for all $a, b \in B$, $a \circ' b=a +' \lambda_a(b)$, an easy computation shows that the only idempotent with respect to the operation $+'$ is $0$ itself.
  The same argument applies if $0=x$. Then  $\left(B, +', \circ'\right)$ would actually be a skew brace, and this is a contradiction because the solution $r_B$ is not bijective.
\end{Ex}

\bigskip

\bigskip 

\section*{Acknowledgements}

\smallskip

A.A., M.M., and P.S. are partially supported by the University of Salento - Department of Mathematics and Physics ``E. De Giorgi”. They are members of GNSAGA (INdAM), and members of the nonprofit association ``AGTA-Advances in Group Theory and Applications". A.A. is supported by a scholarship
financed by the Ministerial Decree no. 118/2023, based on the NRRP - funded by the
European Union - NextGenerationEU - Mission 4.

\bigskip\bigskip

\makeatletter
\renewrobustcmd*{\mkbibemph}{}
\protected\long\def\blx@imc@mkbibemph#1{#1}
\makeatother

\printbibliography

@preamble{
   "\def\cprime{$'$} "
}

@incollection {Dr92,
    AUTHOR = {Drinfel{\cprime}d, V. G.},
     TITLE = {\textit{On some unsolved problems in quantum group theory}},
 BOOKTITLE = {Quantum groups ({L}eningrad, 1990)},
    SERIES = {Lecture Notes in Math.},
    VOLUME = {1510},
     PAGES = {1--8},
 PUBLISHER = {Springer, Berlin},
      YEAR = {1992},
   MRCLASS = {17B37 (16W30 81R50)},
  MRNUMBER = {1183474},
MRREVIEWER = {Yvette Kosmann-Schwarzbach},
       DOI = {https://doi.org/10.1007/BFb0101175},
       URL = {},
}

@article{COJeVaVe22,
title = {\textit{Left non-degenerate set-theoretic solutions of the Yang-Baxter equation and semitrusses}},
AUTHOR = {Colazzo, I. and Jespers, E. and {Van Antwerpen}, A. and Verwimp, C.},
journal = {Journal of Algebra},
volume = {610},
pages = {409-462},
year = {2022},
issn = {0021-8693},
doi = {https://doi.org/10.1016/j.jalgebra.2022.07.019},
}

@book {Ho95,
	AUTHOR = {Howie, J. M.},
	TITLE = {\textit{Fundamentals of semigroup theory}},
	SERIES = {London Mathematical Society Monographs. New Series},
	VOLUME = {12},
	NOTE = {{O}xford Science Publications},
	PUBLISHER = {The Clarendon Press, Oxford University Press, New York},
	YEAR = {1995},
	PAGES = {x+351},
	ISBN = {0-19-851194-9},
	MRCLASS = {20Mxx (20-02)},
	MRNUMBER = {1455373},
	MRREVIEWER = {P. M. Higgins},
}

@article {Rey,
    AUTHOR = {Reynolds, O.},
     TITLE = {\textit{On the dynamical theory of incompressible viscous fluids and
              the determination of the criterion}},
   JOURNAL = {Proc. Roy. Soc. London Ser. A},
  FJOURNAL = {Proceedings of the Royal Society. London. Series A.
              Mathematical, Physical and Engineering Sciences},
    VOLUME = {451},
      YEAR = {1995},
    NUMBER = {1941},
     PAGES = {5--47},
      ISSN = {0962-8444,2053-9169},
   MRCLASS = {76D99 (76-03)},
  MRNUMBER = {1363190},
       DOI = {https://doi.org/10.1098/rspa.1995.0118},
       URL = {},
}

@article{JePi25,
url = {},
title = {\textit{Diagonals of solutions of the {Y}ang--{B}axter equation}},
author = {Jedli\v cka, P. and Pilitowska, A.},
fjournal = {Forum Mathematicum},
journal = {Forum Math.},
doi={https://doi.org/10.1515/forum-2024-0409},
year = {2025},
lastchecked = {2025-11-14}
}

@article{CCSt20-2,
	AUTHOR = {Catino, F. and Colazzo, I. and Stefanelli, P.},
	TITLE = {\textit{The {M}atched {P}roduct of the {S}olutions to the
	{Y}ang--{B}axter {E}quation of {F}inite {O}rder}},
	JOURNAL = {Mediterr. J. Math.},
	FJOURNAL = {Mediterranean Journal of Mathematics},
	VOLUME = {17, 58},
	YEAR = {2020},
	NUMBER = { },
	PAGES = { },
	ISSN = {1660-5446},
	MRCLASS = {16T25 (16N20 16Y99 81R50)},
	MRNUMBER = {4067191},
	DOI = {https://doi.org/10.1007/s00009-020-1483-y},
	URL = {},
}

@article {Brze,
    AUTHOR = {Brzezi\'nski, T.},
     TITLE = {\textit{Trusses: between braces and rings}},
   JOURNAL = {Trans. Amer. Math. Soc.},
  FJOURNAL = {Transactions of the American Mathematical Society},
    VOLUME = {372},
      YEAR = {2019},
    NUMBER = {6},
     PAGES = {4149--4176},
      ISSN = {0002-9947,1088-6850},
   MRCLASS = {16Y99 (16T05)},
  MRNUMBER = {4009388},
MRREVIEWER = {Ilaria\ Colazzo},
       DOI = {https://doi.org/10.1090/tran/7705},
       URL = {},
}

@article {JeAr19,
    AUTHOR = {Jespers, E. and Van Antwerpen, A.},
     TITLE = {\textit{Left semi-braces and solutions of the {Y}ang-{B}axter
              equation}},
   JOURNAL = {Forum Math.},
  FJOURNAL = {Forum Mathematicum},
    VOLUME = {31},
      YEAR = {2019},
    NUMBER = {1},
     PAGES = {241--263},
      ISSN = {0933-7741},
   MRCLASS = {20M25 (16S36 16T25 20E22)},
  MRNUMBER = {3898225},
MRREVIEWER = {Shanghua Zheng},
       DOI = {https://doi.org/10.1515/forum-2018-0059},
}

@article {Ku83,
    AUTHOR = {Kunze, M.},
     TITLE = {\textit{Zappa products}},
   JOURNAL = {Acta Math. Hungar.},
  FJOURNAL = {Acta Mathematica Hungarica},
    VOLUME = {41},
      YEAR = {1983},
    NUMBER = {3-4},
     PAGES = {225--239},
      ISSN = {0236-5294},
   MRCLASS = {20M10 (20M35)},
  MRNUMBER = {703736},
MRREVIEWER = {Heinrich Seidel},
       doi = {https://doi.org/10.1007/BF01961311},
}

@article {Ru07,
    AUTHOR = {Rump, W.},
     TITLE = {\textit{Braces, radical rings, and the quantum {Y}ang-{B}axter
              equation}},
   JOURNAL = {J. Algebra},
  FJOURNAL = {Journal of Algebra},
    VOLUME = {307},
      YEAR = {2007},
    NUMBER = {1},
     PAGES = {153--170},
      ISSN = {0021-8693},
   MRCLASS = {16Y99 (16W30)},
  MRNUMBER = {2278047},
MRREVIEWER = {Gigel Militaru},
       DOI = {https://doi.org/10.1016/j.jalgebra.2006.03.040},
}

@article {GuVe17,
    AUTHOR = {Guarnieri, L. and Vendramin, L.},
     TITLE = {\textit{Skew braces and the {Y}ang-{B}axter equation}},
   JOURNAL = {Math. Comp.},
  FJOURNAL = {Mathematics of Computation},
    VOLUME = {86},
      YEAR = {2017},
    NUMBER = {307},
     PAGES = {2519--2534},
      ISSN = {0025-5718},
   MRCLASS = {16T25 (81R50)},
  MRNUMBER = {3647970},
MRREVIEWER = {Paola Stefanelli},
       doi = {https://doi.org/10.1090/mcom/3161},
}

@article {CaCoSt17,
    AUTHOR = {F. Catino  and I. Colazzo and P. Stefanelli},
     TITLE = {\textit{Semi-braces and the {Y}ang-{B}axter equation}},
   JOURNAL = {J. Algebra},
  FJOURNAL = {Journal of Algebra},
    VOLUME = {483},
      YEAR = {2017},
     PAGES = {163--187},
      ISSN = {0021-8693},
   MRCLASS = {16T25 (16N20 16Y99 81R50)},
  MRNUMBER = {3649817},
MRREVIEWER = {Leandro Vendramin},
       DOI = {https://doi.org/10.1016/j.jalgebra.2017.03.035},
       URL = {},
}

@article {FeSaFa04,
    AUTHOR = {Fenn, R. and Jordan-Santana, M. and Kauffman, L.},
     TITLE = {\textit{Biquandles and virtual links}},
   JOURNAL = {Topology Appl.},
  FJOURNAL = {Topology and its Applications},
    VOLUME = {145},
      YEAR = {2004},
    NUMBER = {1-3},
     PAGES = {157--175},
      ISSN = {0166-8641,1879-3207},
   MRCLASS = {57M27 (57M25)},
  MRNUMBER = {2100870},
MRREVIEWER = {Dale\ P. O. Rolfsen},
       DOI = {https://doi.org/10.1016/j.topol.2004.06.008},
       URL = {},
}

@book {ClPr61,
    AUTHOR = {A. H. Clifford and G. B. Preston},
     TITLE = {\textit{The algebraic theory of semigroups. {V}ol. {I}}},
    SERIES = {Mathematical Surveys, No. 7},
 PUBLISHER = {American Mathematical Society, Providence, R.I.},
      YEAR = {1961},
     PAGES = {xv+224},
   MRCLASS = {20.92},
  MRNUMBER = {0132791},
MRREVIEWER = {\v{S}t. Schwarz},
}

@article {Ya67,
    AUTHOR = {Yang, C. N.},
     TITLE = {\textit{Some exact results for the many-body problem in one dimension
              with repulsive delta-function interaction}},
   JOURNAL = {Phys. Rev. Lett.},
  FJOURNAL = {Physical Review Letters},
    VOLUME = {19},
      YEAR = {1967},
     PAGES = {1312--1315},
      ISSN = {},
   MRCLASS = {81.20},
  MRNUMBER = {261870},
MRREVIEWER = {S. Deser},
       DOI = {https://doi.org/10.1103/PhysRevLett.19.1312},
       URL = {},
}

@article {Ba72,
    AUTHOR = {Baxter, R. J.},
     TITLE = {\textit{Partition function of the eight-vertex lattice model}},
   JOURNAL = {Ann. Physics},
  FJOURNAL = {Annals of Physics},
    VOLUME = {70},
      YEAR = {1972},
     PAGES = {193--228},
      ISSN = {0003-4916},
   MRCLASS = {82.46},
  MRNUMBER = {290733},
MRREVIEWER = {S. Sherman},
       DOI = {https://doi.org/10.1016/0003-4916(72)90335-1},
       URL = {},
}

@article {CaCeSt22,
    AUTHOR = {F. Catino and F. Ced\'{o} and P. Stefanelli},
     TITLE = {\textit{Nilpotency in left semi-braces}},
   JOURNAL = {J. Algebra},
  FJOURNAL = {Journal of Algebra},
    VOLUME = {604},
      YEAR = {2022},
     PAGES = {128--161},
      ISSN = {0021-8693},
   MRCLASS = {16T25 (16Y99 81R50)},
  MRNUMBER = {4413279},
MRREVIEWER = {Leandro Vendramin},
       DOI = {https://doi.org/10.1016/j.jalgebra.2022.04.004},
       URL ={},
}

@article {Aguiar00,
    AUTHOR = {M. Aguiar},
     TITLE = {\textit{Pre-{P}oisson algebras}},
   JOURNAL = {Lett. Math. Phys.},
  FJOURNAL = {Letters in Mathematical Physics},
    VOLUME = {54},
      YEAR = {2000},
    NUMBER = {4},
     PAGES = {263--277},
      ISSN = {0377-9017,1573-0530},
   MRCLASS = {17B63 (18D50)},
  MRNUMBER = {1846958},
MRREVIEWER = {Thierry\ Lambre},
       DOI = {https://doi.org/10.1023/A:1010818119040},
}

@article {HuhuXing25,
    AUTHOR = {H. Zhang and X. Gao},
     TITLE = {\textit{Averaging operators on groups and {H}opf algebras}},
   JOURNAL = {Comm. Algebra},
  FJOURNAL = {Communications in Algebra},
    VOLUME = {53},
      YEAR = {2025},
    NUMBER = {12},
     PAGES = {5077--5100},
      ISSN = {},
   MRCLASS = {22E60 (08B20 16T05 16W99 17B38 17B40)},
  MRNUMBER = {4964932},
       DOI = {https://doi.org/10.1080/00927872.2025.2505071},
       URL = {},
}

@article {DoRySte24,
    AUTHOR = {Doikou, A. and Rybo{\l}owicz, B. and Stefanelli, P.},
     TITLE = {\textit{Quandles as pre-{L}ie skew braces, set-theoretic {H}opf algebras \& universal {$\mathcal{R}$}-matrices}},
   JOURNAL = {{\textit J. Phys. A}},
    VOLUME = {57},
      YEAR = {2024},
    NUMBER = {40},
     PAGES = {Paper No. 405203, 35},
     DOI = {https://doi.org/10.1088/1751-8121/ad7769}
}

@article {CCSt20-1,
    AUTHOR = {Catino, F. and Colazzo, I. and Stefanelli, P.},
     TITLE = {\textit{The matched product of set-theoretical solutions of the
              {Y}ang-{B}axter equation}},
   JOURNAL = {J. Pure Appl. Algebra},
  FJOURNAL = {Journal of Pure and Applied Algebra},
    VOLUME = {224},
      YEAR = {2020},
    NUMBER = {3},
     PAGES = {1173--1194},
      ISSN = {0022-4049,1873-1376},
   MRCLASS = {16T25 (16N20 16Y99 81R50)},
  MRNUMBER = {4009573},
MRREVIEWER = {Jo\~ao\ Matheus Jury Giraldi},
       DOI = {https://doi.org/10.1016/j.jpaa.2019.07.012},
       URL = {},
}

@article {Brz18,
    AUTHOR = {Brzezi\'nski, T.},
     TITLE = {\textit{Towards semi-trusses}},
   JOURNAL = {Rev. Roumaine Math. Pures Appl.},
  FJOURNAL = {Revue Roumaine de Math\'ematiques Pures et Appliqu\'ees.
              Romanian Journal of Pure and Applied Mathematics},
    VOLUME = {63},
      YEAR = {2018},
    NUMBER = {2},
     PAGES = {75--89},
      ISSN = {0035-3965},
   MRCLASS = {20M18 (16T05 16Y99)},
  MRNUMBER = {3812011},
  URL = {https://imar.ro/journals/Revue_Mathematique/pdfs/2018/2/2.pdf}
}

@article{AlSte25,
    author = {Albano, A. and Stefanelli, P.},
    title = {\textit{Generalized digroups, di-skew braces, and solutions of the set-theoretic Yang–Baxter equation}},
    journal = {Semigroup Forum},
    year = {2025},
    doi = {https://doi.org/10.1007/s00233-025-10606-2}
}

\bigskip \bigskip

\end{document}